\documentclass[oneside, 12pt]{amsart}
\usepackage{amscd}
\usepackage{amssymb}
\usepackage{amsfonts}
\usepackage{remark}
\usepackage{enumerate}
\usepackage{xypic}
\xyoption{curve} 
\usepackage{hyperref} 

\numberwithin{equation}{section}
\newtheorem{thm}{Theorem}[section]
\newtheorem{theorem}[thm]{Theorem}
\newtheorem{lemma}[thm]{Lemma}
\newtheorem{corollary}[thm]{Corollary}
\newtheorem{proposition}[thm]{Proposition}
\newremark{definition}[thm]{Definition}
\newremark{remark}[thm]{Remark}
\newremark{example}[thm]{Example}

\newcommand\Div{\operatorname{Div}}
\newcommand{\Proj}{\operatorname{Proj}}
\newcommand{\Spec}{\operatorname{Spec}}
\newcommand{\Pic}{\mbox{\rm Pic}\kern 1pt}
\newcommand{\cO}{{\mathcal O}}
\newcommand{\m}{{\mathfrak m}}

\newcommand{\Frac}{\operatorname{Frac}}
\newcommand{\codim}{{\operatorname{codim}}}
\newcommand{\Xm}{{X_{\mathrm{min}}}}
\newcommand{\sm}{{\mathrm{sm}}}
\newcommand{\sh}{{\mathrm{sh}}}
\newcommand{\sep}{{\mathrm{sep}}}
\newcommand{\Mor}{{\operatorname{Mor}}}

\begin{document}

\date{\today}

\title{N\'eron models of algebraic curves} 

\author{Qing Liu}
\author{Jilong Tong}

\address{Universit\'e de Bordeaux 1, Institut de Math\'ematiques de 
Bordeaux, CNRS UMR 5251, 33405 Talence, France} 
\email{Qing.Liu@math.u-bordeaux1.fr}
\email{Jilong.Tong@math.u-bordeaux1.fr} 


\begin{abstract} Let $S$ be a Dedekind scheme with field of functions $K$. We show that if $X_K$ is a smooth connected proper curve of positive genus over $K$, then it admits a N\'eron model over $S$, {\it i.e.}, a smooth separated  model of finite type satisfying the usual N\'eron mapping property. It is given by the smooth locus of the minimal proper regular model of $X_K$ over $S$, as in the case of elliptic curves. When $S$ is excellent, a similar result holds for connected smooth affine curves different from the affine line, with locally finite type N\'eron models.
\end{abstract} 

\dedicatory{Dedicated to Michel Raynaud on 
the occasion of his seventy-fifth birthday} 


\maketitle 

\begin{section}{Introduction}

In 1964, A. N\'eron \cite{Neron} introduced the notion 
of N\'eron models (see Definition~\ref{def-nm}) of abelian varieties over the 
fraction field of a Dedekind domain, and proved the existence of
these models (see the introduction of 
\cite{BLR} for a detailed presentation). Since then, this notion 
has been generalized to smooth commutative algebraic groups and to torsors under these groups (see \cite{BLR}, \S 6.5). In this work we investigate the case of smooth proper or affine curves. 

Let $S$ be a \emph{Dedekind scheme}, that is, a {noetherian regular connected scheme of dimension $1$}. Let $K=K(S)$ be its field of functions. Let $X_K$ be a smooth connected curve over $K$. When $X_K$ is proper of positive genus, a canonical smooth
model is the smooth locus $X_\sm$ of the minimal proper 
regular model of $X_K$ over $S$. When $X_K$ is an
elliptic curve, it is well known that $X_\sm$ is the N\'eron model of $X_K$ (see \cite{Neron} or \cite{Liu}, 10.2.14). The first main result of this work is a generalization of the latter fact to higher genus. 

\begin{theorem}\label{pj-nm-it}
{\rm (Theorem \ref{pj-nm})} Let $X_K$ be a proper smooth connected curve of positive genus over $K$. Then the smooth locus $X_\sm$ of the minimal proper regular model of $X_K$ over $S$ is the N\'eron model of $X_K$.   
\end{theorem}

When $S$ is excellent, {we actually prove a slightly more general result: for a regular proper connected curve $X_K/K$ of positive arithmetic genus, the smooth locus of $X_K$ admits a N\'eron model over $S$, equal to the smooth locus of the minimal proper regular model of $X_K$ over $S$.}
 
As an immediate consequence of Theorem~\ref{pj-nm-it}, we have the next corollary.  

\begin{corollary} Let $X_K$ be as in Theorem~\ref{pj-nm-it}. 
Let $Y$ be a smooth scheme over $S$ and let $f_K\colon Y_K\to X_K$ be a morphism of $K$-schemes. Then
 \begin{enumerate}[{\rm (1)}] 
  \item the morphism $f_K$ extends {uniquely} to a morphism of $S$-schemes $Y\to X_\sm$; 
  \item {\rm (Corollary~\ref{cov-gr})} if $Y$ is proper over $S$ ({\it i.e.}, $Y_K$ has good reduction over
$S$) and $f_{K}$ is dominant, then $X_\sm$ is proper over $S$ and $X_K$ 
has good reduction over $S$. 
  \end{enumerate}
\end{corollary}

In the second part of this paper (\S \ref{af-nm-s}--\S
\ref{nm-rat-gl}), we consider N\'eron lft-models (N\'eron model
locally of finite type) of smooth affine curves. 
The main result of this second part is 

\begin{theorem}{\rm (Theorem \ref{nm-af})} \label{nm-af-i} 
Suppose $S$ is excellent. Let $U_K$ be an affine smooth and geometrically 
connected curve over $K$, not isomorphic to $\mathbb A^1_K$. Then $U_K$ 
admits a N\'eron lft-model $U$ over $S$. 
\end{theorem} 

Note that in general, the scheme $U$ in the above theorem is not of finite type. However, necessary and sufficient conditions (in terms of the points at infinity of
$U_K$) can be found in {Proposition}~\ref{af-nm-ft} to insure that $U$ is of finite type over $S$. 

The paper is organized as follows. Some basic properties of N\'eron
lft-models are assembled in \S \ref{basic-p}. In \S \ref{image-sm}, we
prove a crucial technical result (Proposition~\ref{fY-reg}) on the
image of a morphism  $f \colon Y\to X$ from a smooth $S$-scheme to a
normal relative curve over $S$. In \S \ref{pj-nm-s}, we prove the main theorem \ref{pj-nm} on the existence of N\'eron models for proper smooth  
curves of positive genus. In \S \ref{af-nm-s}, we study the N\'eron lft-models of
open subschemes of a curve having already a N\'eron lft-model 
(Theorem~\ref{nm-open-st}). \S \ref{rat-nm} is devoted to the existence of the 
N\'eron model of some special affine open subschemes of a smooth conic 
over local Dedekind schemes $S$. Finally we 
prove Theorem~\ref{nm-af-i} in \S \ref{nm-rat-gl}. 
\smallskip

The present work grew from a partial answer to a question 
asked by an anonymous
poster at {\tt mathoverflow.net/questions/110359/}, 
on the existence of N\'eron models of projective curves. {We would like to thank A. Javanpeykar and M. Raynaud for their interests in this work and especially for their comments which improve the presentation of the paper.}
\bigskip

\noindent {\bf Notation:} In all this paper, unless explicitly mentioned, the letter $S$ denotes a \emph{Dedekind scheme} (that is, a noetherian regular 
connected scheme of dimension $1$), $K$ denotes its field of functions $K(S)$.
Symbols {such} as $X_K, Y_K, U_K$ usually denote a scheme over $K$. On the other
  hand, for any $S$-scheme $X$, $X_K$ also denotes the generic fiber of $X$.
\end{section} 

\begin{section}{Basic properties} \label{basic-p} 

Let $S, K$ be as above. 

\begin{definition} Let $X_K$ be a separated algebraic variety ({\it i.e.} 
separated scheme of finite type) over $K$. A \emph{model of $X_K$ over $S$} 
is a {\bf locally finite type}, separated and flat\footnote{In this work, we do not require the models to be faithfully flat.} scheme over $S$ endowed with an isomorphism from its generic fiber to $X_K$. 
 \end{definition}

\begin{definition}[\cite{BLR} 10.1/1, 1.2/1]\label{def-nm} Let $X_K$ be a separated 
smooth algebraic variety over $K$. A \emph{N\'eron lft-model of 
$X_K$ over $S$} or an \emph{$S$-N\'eron lft-model of $X_K$} 
is a smooth model $X$ of $X_K$ over $S$ 
satisfying the following universal property, called 
{\emph{N\'eron mapping property}}:
\smallskip 

\leftskip 12pt \noindent 
for any smooth scheme $Y\to S$, the canonical map (once the generic
fiber of $X$ is identified with $X_K$) 
\begin{equation}\label{eq.def-neron}
\Mor_S(Y, X) \to \Mor_K(Y_K, X_K)
\end{equation}
is a bijection. 

\smallskip 
\leftskip 0pt 
 \noindent A {N\'eron lft-model} of finite type is called a \emph{N\'eron model}. 
\end{definition}

\begin{remark} \label{univ}
\begin{enumerate}
\item The universal property above implies the uniqueness
(up to a unique isomorphism) of the N\'eron lft-model if it exists. 
\item Let $X$ be a model of $X_K$. As $X$ is separated, the map \eqref{eq.def-neron} is always injective. So
it is enough to check the surjectivity for the N\'eron mapping property. By the injectivity, when $S$ is local, it is also enough to check the surjectivity with $Y$ smooth of finite type having irreducible fibers over $S$. 
\item Let $X_K$ be a separated smooth algebraic variety over $K$.
If each connected component of $X_K$ admits an $S$-N\'eron lft-model, then
$X_K$ has also an $S$-N\'eron lft-model given by the disjoint union of
the $S$-N\'eron lft-models of the connected components. 
{This holds similarly for N\'eron models}.
\end{enumerate}
\end{remark}

\begin{proposition} \label{nm-gen}
Let $X$ be an $S$-model of $X_K$. Let $S'\to S$ be a 
morphism of Dedekind schemes. {Denote by $K'$ the function field of $S'$.}
\begin{enumerate}[{\rm (1)}] 
\item Assume that the morphism $S'\rightarrow S$ is faithfully flat, and that $X_{S'}:=X\times_S S'$ is the N\'eron lft-model (resp. N\'eron model) of $X_{K'}:=X\times_{\Spec(K)}\Spec(K')$ over $S'$. Then $X$ is the N\'eron lft-model (resp. N\'eron model) of $X_K$ over $S$. 

\item Assume that $S'\rightarrow S$ can be written as a filtered inverse limit of {affine} smooth schemes of finite type over $S$, and that $X$ is the N\'eron lft-model (resp. N\'eron model) of $X_K$ over $S$. Then the base change $X_{S'}$ is the N\'eron lft-model (resp. N\'eron model) of $X_{K'}$ over $S'$.

\item Assume that $S'\rightarrow
    S$ is an extension of local Dedekind schemes of ramification index
    $1$\footnote{see \cite{BLR}, 3.6/1 for the definition, the residue
    extension is required to be (not necessarily algebraic) separable.} with $S$ \emph{excellent}, and that $X$ is the N\'eron
lft-model over $S$. Then, $X_{S'}$ is the N\'eron lft-model over
$S'$.
\end{enumerate}
  \end{proposition}

\begin{proof} (1) This is a direct application of fpqc descent, so we omit the details here. 

(2) Let $Y'$ be a smooth $S'$-scheme, and let $f_{K'}\colon Y_{K'}'\rightarrow X_{K'}$ be a morphism of $K'$-schemes. Without loss of generality, we may and do assume that the scheme $Y'$ is quasi-compact, hence is of finite presentation over $S'$. Consequently, 
the $S'$-scheme $Y'$ {(resp. the morphism $f_{K'}\colon Y_{K'}'\rightarrow X_{K'}$)} descends to an $S_0$-scheme $Y_0$ {(resp. to a morphism of $S_{0,K}$-schemes $f_{0,K}\colon Y_{0}\times_{S_0} S_{0,K}\rightarrow X_{S_0}\times_{S_0} S_{0,K}$)}
for some {(affine)} smooth morphism $S_0\to S$ (\cite{EGA}, \uppercase\expandafter{\romannumeral4} 8.8.2). Thus, we only need to prove that the morphism $f_{0,K}$ can be extended to a morphism $f_0\colon Y_0\rightarrow X\times_S S_0$. Consider the composite $Y_0\rightarrow S_0\rightarrow S$, which defines a smooth $S$-scheme. The morphism $f_{0,K}$ gives a morphism of $K$-schemes from $Y_0\times_{S}\Spec(K)$ to $X_K$. As $X$ is the {$S$-}N\'eron lft-model of $X_K$, the N\'eron mapping property implies a morphism of $S$-schemes $g\colon Y_0\rightarrow X$ making the following diagram commutative:
\[
\xymatrix{Y_{0}\ar[rd] \ar@/^1pc/[rr]^g\ar@{.>}[r]_{\exists f_0}& X_{S_0}\ar[d]\ar[r] & X\ar[d] \\ & S_0\ar[r]& S}
\] 
We obtain a morphism $f_0 \colon Y_0\rightarrow X_{S_0}$ of
$S_0$-schemes extending $f_{0,K}$, as required. 
{If $X$ is the $S$-N\'eron model of $X_K$, then $X_{S'}$ is of finite type over $S'$ and is the $S'$-N\'eron model of $X_{K'}$.}

(3) Write $S=\mathrm{Spec}(R)$ and $S'=\mathrm{Spec}(R')$. As $S$ is excellent, by \cite{BLR}, 3.6/2, the morphism $R\rightarrow R'$ is regular. Consequently, by a result of N\'eron ({\cite{BLR}, 3.6/8, see also \cite{Popescu}, 2.5 for a more general statement}), $R'$ can be written as a filtered inductive limit of smooth $R$-algebras of finite type. Therefore, $X_{S'}$ is the N\'eron lft-model over $S'$ by (2). 
\end{proof} 

\begin{corollary}\label{neron-l2g} 
Let $S$ be a Dedekind scheme with field of functions $K$, and let $X$
be an $S$-scheme locally of finite type {(resp. an $S$-scheme of finite type)} with smooth finite type generic fiber $X_K$. Then $X$ is the $S$-N\'eron lft-model of $X_K$ {(resp. $S$-N\'eron model of $X_K$)} if and only if for any closed point $s\in S$, $X\times_S \Spec(\mathcal{O}_{S,s})$ is the $\Spec(\cO_{S,s})$-N\'eron lft-model of $X_K$ {(resp. $\Spec(\cO_{S,s})$-N\'eron model of $X_K$)}. 
\end{corollary}

\begin{proof} If $X/S$ is of finite type, then this corollary is proved in \cite{BLR}, 1.2/4. The same proof applies in the locally finite type case. For the convenience of the readers, we reproduce the proof in this case here. The direct implication follows from
  Proposition~\ref{nm-gen}~(2). Conversely, if $X\times_S
  \mathrm{Spec}(\cO_{S,s})$ is the N\'eron lft-model of $X_K$ over
  $\mathrm{Spec}(\cO_{S,s})$ for any closed point $s\in S$, it follows that $X/S$ is smooth and separated (\cite{EGA}, IV.8.10.5(v)). It remains to check the N\'eron mapping property for $X/S$ (Definition~\ref{def-nm}). Consider $Y$ a smooth $S$-scheme, and $f_K\colon Y_K\rightarrow X_K$ a morphism of
  $K$-schemes, we need to extend $f_K$ to a morphism of $S$-schemes from $Y$ to $X$. 
We may assume that $Y$ is quasi-compact, hence is of finite
presentation over $S$. Now, for any closed point $s\in S$, our
assumptions imply that one can find an extension $f_{\cO_{S,s}}$ of
$f_K$ over $\Spec(\cO_{S,s})$. As $Y/S$ is of finite presentation and
$X/S$ is locally of finite presentation, we may descend
$f_{\cO_{S,s}}$ to a morphism $f_s\colon Y\times_S V\rightarrow
X\times_{S}V$ over some open neighborhood $V\subseteq S$ of $s$
(\cite{EGA}, \uppercase\expandafter{\romannumeral4} 8.8.2). As $X/S$ is
separated, and $Y/S$ is flat, these local extensions of $f_K$ are
compatible with each other. As a result, they can be glued together to
a morphism $f\colon Y\rightarrow S$ extending $f_K$, as desired.  
\end{proof}

The next lemma says that we can restrict ourselves to geometrically connected varieties $X_K$ without loss of generality. 
Denote by $\overline{K}$ an algebraic closure of $K$. {For any closed point $s\in S$ we denote by $K_{s}^{\sh}$ the fraction field of the strict henselization of the local ring $\cO_{S,s}$.}

\begin{lemma} \label{connectedness} Let $X_K$ be a separated smooth connected
variety over $K$. Let $K'=K(X_K)\cap \overline{K}$ be the field of constants of $K(X_K)$,
let $S'$ be the integral closure of $S$ in $K'$ and let $T\subseteq S'$ be the
\'etale locus of $S'\to S$. Then 
\begin{enumerate}[{\rm (1)}] 
\item $X_K$ is canonically a separated smooth 
and geometrically connected variety over $K'$;
\item $X_K/K$ admits an $S$-N\'eron lft-model (resp. $S$-N\'eron model) if and only if $\dim T=0$, or $\dim T=1$ and $X_K/K'$ admits a $T$-N\'eron lft-model (resp. $T$-N\'eron model). 
\end{enumerate}
\end{lemma}

\begin{proof} (1) As $K'\subseteq \cO_{X_K}(X_K)\subseteq K(X_K)$ by the 
normality of $X_K$, the latter has a canonical structure of
$K'$-variety. As $K'$ is algebraically 
closed in $K(X_K)$, $X_K/K'$ is geometrically connected. As 
$X_K\times_{\Spec(K')}\Spec(\overline{K})$ is a connected component of  
$X_K\times_{\Spec(K)} \Spec(\overline{K})$, $X_K/K'$ is separated 
and smooth.  

(2) Note that because $X_K$ is smooth over $K$, $K'/K$ is separable, so $S'\to S$ is finite. If $\dim T=0$, then $S$ is semi-local and $S'\to S$ is ramified at all closed points. This implies that for all closed points $s\in S$, $X_K$ has no $K^{\sh}_s$-point, so 
$X_K$ is its own $S$-N\'eron model. Suppose now that $\dim T=1$ and that $X_K/K'$ has a $T$-N\'eron 
lft-model $X$. Then the composition with 
$T\to S$ makes $X$ into a smooth separated $S$-scheme, with generic fiber $X_K/K$. Let us check that it satisfies 
the N\'eron mapping property. Let $Y\to S$ be a smooth scheme 
and let $f_K\colon Y_K\to X_K$ be a $K$-morphism. Then $Y_K\to \Spec(K)$ also 
factors through $Y_K\to \Spec(K')$ via $X_K\to\Spec(K')$. In particular, $f_K$ is a $K'$-morphism. On the other hand, as $Y$ is normal, $Y\to S$ factors through $Y\to S'$. The latter has image in $T$ by Corollary~\ref{sm-to-sm} and makes $Y$ a smooth $T$-scheme. 
So $f_K$ extends to a $T$-morphism $f: Y\to X$, which is 
{\it a fortiori} an $S$-morphism. 

Conversely, if $X_K/K$ has an $S$-N\'eron lft-model $X$, the above {arguments} show that $X$ is canonically a smooth separated $T$-scheme, and the 
generic fiber of $X_T$ is nothing but $X_K$ {viewed as a scheme over $K'$}. If $\dim T=1$, the 
N\'eron mapping property of $X\to T$ is immediate to verify. 

Finally, as $T\to S$ is of finite type and separated, 
$X\to T$ is of finite type if and only if $X\to S$ is of finite type. 
\end{proof}

\begin{corollary}  Let $X_K$ be a smooth $K$-variety of
  dimension zero. Then $X_K$ admits a N\'eron model over $S$. 
\end{corollary}

\begin{proof} By Proposition~\ref{univ} (3) and
  Lemma~\ref{connectedness}, we can suppose $X_K/K$ is geometrically
    connected, smooth of dimension $0$. Then $X_K=\Spec(K)$, and $S$ is clearly the N\'eron
model of $X_K$ over $S$. 
\end{proof}

The following proposition will be used to see when a N\'eron lft-model
is a N\'eron model ({\it e.g.} in Proposition~\ref{af-nm-ft}). 

\begin{proposition}\label{nm-ft} Let $X\to S$ be a 
separated morphism locally of finite type, such that 
$X\times_S \Spec(\mathcal{O}_{S,s})$ 
is of finite type for all $s$. 
\begin{enumerate}[{\rm (1)}] 
\item If $X$ is irreducible and $X_K$ is proper over $K$, then $X$ is of finite type over
  $S$.
\item If $X_K$ is affine, $X_s$ is irreducible for all $s\in S$ and $S$ is 
excellent, then $X$ is of finite type over $S$.
\item If $X$ is of finite type and if $X_K/K$ is geometrically
connected, then $X_s$ is geometrically connected for all $s$ in
some dense open subset of $S$. 
\end{enumerate}
\end{proposition}

\begin{proof} (1) Let $U$ be a quasi-compact open subset of $X$ such that $X_K\subseteq U$, and let ${U'}$ be a Nagata compactification of $U\to S$. Then $U$ and ${U'}$ are of finite type over $S$ sharing the same generic fiber. So there exists a dense open subset $V$ of $S$ such that $U\times_S V\cong {U'}\times_S V$ is proper over $V$. 
The inclusion $U\times_S V\to X\times_S V$ is then open and closed. As $X\times_S V$ is irreducible, $X\times_S V=U\times_S V$, {thus} is of finite type over $V$. The remaining part $X\times_S (S\setminus V)$ is a finite union of quasi-compact subsets, so $X$ is quasi-compact. 

(2) As in (1), to prove $X$ is quasi-compact, we are allowed to shrink
$S$. 
Let $U$ be a quasi-compact open subset of $X$ containing $X_K$, and
let $W$ be an affine finite type $S$-scheme such that $W_K=X_K$. 
As $U, W$ are of finite type over the noetherian scheme $S$ with the same generic fiber, shrinking $S$ if necessary, we can suppose that $U=W$ is affine, and that $U\to S$ 
is surjective (with $S$ affine). We claim that $X=U$. 
Let $U'$ be any affine open subset of $X$ not contained in $U$ and
$F:=U'\setminus (U\cap U')$. Then $F\neq \emptyset$, so it has pure codimension $1$ in $U'$ because $U\cap U'$ is affine (see \cite{Liu}, Exercise 4.1.15. The hypothesis $S$ excellent implies that the normalization map of $U'$ is finite). As $F_K=\emptyset$, $F$ is then a finite union of vertical divisors, which is impossible since $U_s$ is dense in $X_s$. Consequently, $X=U$ is of finite type over $S$. 

(3) This follows from \cite{EGA}, IV.9.7.7 after noticing that, as $S$
is irreducible, a dense locally constructible subset of $S$ 
contains an open dense
subset. 
\end{proof}

\begin{remark} In general, the condition 
$X\times_S \Spec(\cO_{S,s})$ of finite type over $\cO_{S,s}$ for all $s$ 
is not sufficient to conclude that $X$ is of finite type over $S$, as an example of Oesterl\'e (\cite{BLR}, 10.1/11) shows. See also Remark~\ref{not-ft}.
\end{remark}
\end{section}

\begin{section}{Image of smooth schemes} \label{image-sm}

Let $f\colon Y\to X$ be a morphism of $S$-schemes. In this section, we study geometric properties of $X$ at the points of $f(Y)$, when $Y$ is smooth over $S$. The main result is Proposition~\ref{fY-sm} which states that $X$ is smooth at points of $f(Y)$ under some mild hypothesis. Its Corollary~\ref{fY-reg} is a principal ingredient of the proof of Theorem~\ref{pj-nm}. 

Let $Y$ be a scheme. For any morphism locally of finite type $Z\to Y$, we denote by $\sm(Z/Y)\subseteq Z$ the smooth locus of $Z\to Y$. This is an open subset of $Z$ if $Y$ is locally noetherian. 

\begin{lemma} \label{sm-sect} 
Let $Z\to Y$ be a morphism locally of finite type. Suppose that $Y, Z$
are locally noetherian and {regular}. Then for any section 
$\sigma \colon Y\to Z$, the image $\sigma(Y)$ is contained in the
smooth locus $\sm(Z/Y)$ of $Z\to Y$. 
\end{lemma}

\begin{proof} See {\cite{BLR}, 3.1/2}. \end{proof} 

\begin{corollary} \label{sm-to-sm}
Let $S$ be a locally noetherian regular scheme. Let $f\colon Y\to X$ be a morphism between two $S$-schemes locally of finite type. Suppose that $Y$ is smooth over $S$, and that $X$ is regular. Then $f(Y)$ is contained in the smooth locus of $X\to S$. 
\end{corollary}

\begin{proof}
Consider the $Y$-scheme $Z:=X\times_S Y$. Notice that $Z$ is regular, being smooth over $X$. The morphism
$f$ induces a section $\sigma\colon Y\to Z$, $y\mapsto (f(y),y)$,
of the second projection $Z\to Y$. By Lemma~\ref{sm-sect}, 
$$\sigma(Y) \subseteq \sm(Z/Y)=\sm(X/S)\times_S Y,$$ 
hence $f(Y)\subseteq \sm(X/S)$. 
\end{proof}

Corollary~\ref{sm-to-sm} does not hold in general if we remove the 
regularity hypothesis on $X$. However, in the situation of relative
curves, we can weaken the regularity hypothesis to the normality of
$X$ (Proposition~\ref{fY-sm}). We first prove some preliminary results. 

\begin{lemma} \label{qf-irr} Let $S$ be an irreducible locally
  noetherian scheme, and let $X, Y$ be irreducible flat $S$-schemes 
locally of finite type. Let $f\colon Y\to X$ be a dominant $S$-morphism. 
Let $s\in S$. Suppose that $f$ is quasi-finite at some point $y_0\in Y_s$, and that $Y_s$ is irreducible at $y_0$. Then 
$X_s$ is irreducible at $x_0:=f(y_0)$.  
\end{lemma}

\begin{proof} The property is local on $X$ and $Y$. 
In particular, shrinking {$X$} and $Y$ if necessary,
we can suppose that $f\colon Y\to X$ is quasi-finite {and separated}. Thus $f$ can be factorized as $Y\to \overline{Y} \to X$ with an open (dense) immersion followed by a finite surjective morphism {(\cite{EGA}, IV.8.12.6)}. Let $Z_1, Z_2$ be two 
irreducible components of $X_s$ passing through $x_0$.
By the {going-down} property of $\overline{Y}\to X$ (\cite{Mat}, 5.E.(v)), there exist irreducible closed subschemes $F_1, F_2$ of {$\overline{Y}$}
passing through $y_0$ such that the induced maps $F_i\to Z_i$ ($i=1,2$) are finite and dominant (thus surjective). 
Let $\eta$ be the generic point of $S$. As $X, Y$ are irreducible and
flat over $S$, both $X_s$ and $Y_s$ are equidimensional of dimension 
$\dim X_\eta=\dim Y_\eta$ (\cite{EGA}, IV.14.2.3). Consequently,
$$\dim F_i=\dim Z_i=\dim_{x_0} X_s=\dim X_\eta=\dim Y_\eta=\dim_{y_0} Y_s$$ 
and {$F_i\cap Y_s=Y_s$}. Therefore $Z_1=Z_2$ and $X_s$ is irreducible
at $x_0$. 
\end{proof} 

\begin{lemma} \label{bertini} Let $S$ be a locally noetherian scheme,
  let $X, Y$ be two $S$-schemes locally of finite type, and let 
  $f\colon Y\to X$ be a morphism of $S$-schemes. Consider $s\in S$, and $y_0\in Y_s$ a closed point of $Y_s$.  Let 
$x_0=f(y_0)$. Suppose that $Y_s$ is regular at $y_0$ and 
$$\dim_{y_0} Y_s>1 \quad \text{\rm and }  \ \codim_{y_0} (f^{-1}(x_0), Y_s)>0.$$ 
(The second inequality means that $f$ is non-constant 
on the irreducible component of $Y_s$ containing $y_0$). 
Then there exists a subscheme $Z$ of $Y$ passing 
through $y_0$ such that 
$$\dim_{y_0} Z_s < \dim_{y_0} Y_s \quad 
\text{\rm and }  \ \codim_{y_0} (f^{-1}(x_0)\cap Z, Z_s)>0, $$ 
and $Z_s$ is regular at $y_0$. If furthermore $Y$ is regular (resp. if $Y\to S$ 
is flat) at $y_0$, we can assume that the same property holds for $Z$. 
\end{lemma}

\begin{proof} We construct $Z$ locally at $y_0$ as a hypersurface 
defined by some $u\in \m_{y_0}\cO_{Y, y_0}$ which must avoid some ideals of $\cO_{Y,y_0}$. We notice the following facts: 
\begin{enumerate}[{\rm (i)}]
\item Let $\Gamma_1, \dots, \Gamma_n$ be the irreducible components of $f^{-1}(x_0)$ of codimension $1$ in $Y_s$ passing through $y_0$. 
Locally at $y_0$, each $\Gamma_i$ is defined by a prime principal ideal 
$\bar{t}_i\cO_{Y_s, y_0}\subseteq \m_{y_0}\cO_{Y_s,y_0}$ because $Y_s$ is regular at 
$y_0$. As $\dim_{y_0} Y_s>1$, we have 
$\m_{y_0}\cO_{Y_s,y_0}\not\subseteq \bar{t}_i\cO_{Y_s,y_0}$; 
\item We have 
$\m_{y_0} \cO_{Y_s, y_0} \not\subseteq \m_{y_0}^2\cO_{Y_s,y_0}$ 
because $Y_s$ has positive dimension at $y_0$. 
\end{enumerate}
By prime avoidance lemma (\cite{Mat}, 1.B), there exists 
$$\bar{u}\in \m_{y_0} \cO_{Y_s, y_0} \setminus 
\left(\m_{y_0}^2\cO_{Y_s,y_0} \cup (\cup_{i\le n} \bar{t}_i\cO_{Y_s, y_0})\right).$$
Lift $\bar{u}$ to some $u\in \m_{y_0} \cO_{Y, y_0}$ and 
let $Z:=V(u)$ be the subscheme of $Y$ defined in some
open neighborhood of $y_0$. Then:  
\begin{enumerate}[{\rm (1)}] 
\item $Z_s$ is regular at $y_0$; if $Y$ is regular (resp. if $Y\to S$
  is flat) at $y_0$, then the same holds for $Z$ because 
$u\notin \m_{y_0}^2\cO_{Y,y_0}$ and $\bar{u}$ is not a zero divisor); 
\item $\dim_{y_0} Z_s < \dim_{y_0} Y_s$; 
\item and $\codim_{y_0}(f^{-1}(x_0)\cap Z, Z_s)>0$, because 
otherwise $V(\bar{u})$ would be contained in, hence equal to, 
some irreducible component of $f^{-1}(x_0)$, so
$\bar{u}\in \bar{t}_i\cO_{Y_s, y_0}$ for some $i\le
n$. Contradiction. 
\end{enumerate}
Therefore $Z$ satisfies the desired properties. 
\end{proof}

\begin{lemma} \label{sp-int} Let $S$ be a Dedekind scheme. 
Let $X$ be a normal relative curve over $S$,\footnote{By a \emph{relative
curve} over $S$, we mean a flat, locally finite type $S$-scheme with
generic fiber of dimension $1$.} {and let $Y$ be a regular scheme which is flat
and locally of {finite type} over $S$.} Suppose $Y_s$ is
regular. Let $f\colon Y\rightarrow X$ be 
an $S$-morphism, and $x_0=f(y_0)$ for some $y_0\in Y_s$. Suppose that 
$$\codim_{y_0}(f^{-1}(x_0), Y_s)>0.$$
Then $X_s$ is irreducible and reduced at $x_0$. If $Y_s$ is 
geometrically reduced in a neighborhood of $y_0$, then $X_s$
is geometrically reduced in a neighborhood of $x_0$. 
\end{lemma}

\begin{proof} We can suppose $X, Y$ are integral. 
Using repeatedly Lemma \ref{bertini}, we find a subscheme 
$Z$ of $Y$ containing $y_0$, flat over $S$, such that $Z_s$ is 
irreducible and regular of dimension $1$ and 
$f|_{Z_s}$ is non-constant.   
This implies that $f|_{Z_s}$ is quasi-finite
and $f|_Z$ is dominant. By Lemma \ref{qf-irr}, 
$X_s$ is irreducible at $x_0$. Shrinking $X$ and $Y$ if necessary,
we can suppose $X_s$ is irreducible. As $f|_{Z_s}$ is non-constant, 
$Y_s\to X_s$, and hence $Y\to X$, are dominant. Thus $X_s$ is reduced at its generic point $\xi$ because the ramification index of $\cO_{S,s}\to \cO_{X,\xi}$
is $1$. 
But $X$ is normal, hence $S_2$, $X_s$ is $S_1$. This implies that
$X_s$ is reduced and $Y_s\to X_s$ is scheme-theoretically dominant. 
Then the same property holds over $\overline{k(s)}$, which implies that 
$X_s$ is geometrically reduced if $Y_s$ is geometrically reduced. 
\end{proof}

\begin{proposition} \label{fY-sm} Let $S$ be a Dedekind scheme, and
  let $X$ be a normal relative curve over $S$ with 
smooth generic fiber. Let $f \colon Y\to X$ be a morphism 
with $Y$ smooth and $Y_s$ irreducible for some {closed point} $s\in S$. 
Then either $f(Y_s)$ is one point, or $X$ is smooth at every point 
of $f(Y_s)$. 
\end{proposition}

\begin{proof} {We may assume that $S=\mathrm{Spec}(R)$ is 
local, and that} $f(Y_s)$ is not one point. 
Then for all $y_0\in Y_s$ and $x_0:=f(y_0)$, we have 
$\codim_{y_0}(f^{-1}(x_0), Y_s)>0$.  
Let $R'$ be any discrete valuation ring dominating
$R$. Then $X_{R'}:=X\otimes_R R'$ 
has smooth generic fiber, and its special fiber is reduced 
at any point $x_0'$ lying over $x_0$ by Lemma \ref{sp-int}. 
So $X_{R'}$ is normal at $x_0'$. Therefore, to prove $X$ is smooth at
$x_0$, 
we can enlarge $R$ and suppose it is complete with algebraically closed residue field. Using Lemma~\ref{bertini}, we can suppose that $Y$ is a relative smooth curve. Then $f$ is quasi-finite, and $\Spec\widehat{\cO}_{Y,y_0} \to \Spec\widehat{\cO}_{X,x_0}$ is finite. 
By a result of Raynaud (\cite{Ray}, Appendice, p. 195), $X$ is smooth
at $x_0$. 
\end{proof}

\begin{remark}\label{mini-desing} {
Let $X$ be an integral relative curve over $S$ with smooth generic 
fiber. Then $X$ admits a minimal desingularization $X'\to X$,
made of a finite sequence of normalizations and blowing-ups of closed
singular points. 
See {\it e.g.} \cite{Liu}, 8.3.50 and 9.3.32 when $X$ is proper
over $S$. As the construction of $X'$ is local on $X$, and the 
minimal desingularization is unique, 
the same result holds for any integral relative curve $X$ over $S$ 
with smooth generic fiber.} 
\end{remark} 

\begin{corollary} \label{fY-reg}  Let $S$ be a Dedekind scheme, and
    let $X$ be an integral relative curve over $S$. Let
$f\colon  Y\to X$ be an $S$-morphism from a smooth $S$-scheme $Y$ to $X$. 
Let $s\in S$ {be a closed point}. 
\begin{enumerate}[{\rm (1)}] 
\item If $f(Y_s)$ is reduced to one point $x_0\in X_s$, then 
$f$ factors {as} $Y\to \widetilde{X}\to X$, where the
second morphism is the blowing-up of $X$ along the reduced center $x_0$. 
\item Suppose that $X_K$ is smooth, 
$Y$ is irreducible and that $f_K \colon Y_K\to X_K$ is dominant. Let $X'\to X$ be the minimal 
desingularization of $X$ (Remark \ref{mini-desing}). Then $f\colon Y\rightarrow X$ factors through ${X}'\to X$.   
\end{enumerate}
\end{corollary}

\begin{proof} (1) We have to show that $\m_{x_0}\cO_{Y}$ is an
  invertible sheaf of ideals of $\cO_Y$. Let $y\in f^{-1}(x_0)=Y_s$, and
  let $\pi$ be a generator of $\m_s\cO_{S,s}$. We have 
$$\sqrt{\m_{x_0}\cO_{Y, y}}=\sqrt{\pi \cO_{Y, y}}=\pi\cO_{Y,y}$$ 
(because $Y_s$ is reduced). 
As $\pi\cO_{Y,y}\subseteq \m_{x_0}\cO_{Y,y}$, we 
find $\m_{x_0}\cO_{Y,y}=\pi\cO_{Y,y}$. Therefore 
$\m_{x_0}\cO_{Y}$  is an invertible sheaf of ideals, and 
$Y\to X$ factors through $Y\to \widetilde{X}\to X$.  

(2) {As the minimal desingularization commutes with 
restriction to open subsets, and because the property to prove is 
local at $Y$, we can suppose $Y$ is quasi-compact. Then $f(Y)$ is 
contained in a quasi-compact open subset of $X$. Therefore we can 
also suppose $X$ is quasi-compact.}

{As $Y$ is normal and $f_K$ is dominant, $f$ factors through the
normalization of $X$. Furthermore, the normalization map of $X$ is finite
(see \cite{Liu}, 8.3.49(d)). So we can suppose $X$ is normal.} 

Let $F$ be the singular locus of $X$, which is a finite
closed subset of $X_s$. Let $X_1\to X$ be the blowing-up along $F$
(with the reduced structure). If $F\cap f(Y_s)=\emptyset$, then $f$
trivially factors {as} $Y\to X_1\to X$. In general, 
for any $x_0\in F\cap f(Y_s)$, it follows easily
from Proposition~\ref{fY-sm} that $f^{-1}(x_0)$ is a union of irreducible
components of $Y_s$ and, by (1), $f$ factors through
$X_1\to X$. Similarly {as above}, the morphism $Y\to X_1$ factors through the normalization map $X_1'\to X_1$. 
Now we start again with $Y\to X_1'$ and the process will stop 
at {the minimal} desingularization of $X$. 
\end{proof}

\end{section} 

\begin{section}{N\'eron models of proper smooth curves} \label{pj-nm-s} 

Let $S$ be a Dedekind scheme with field of functions $K$. Let $X_K$ be a proper {regular} and connected curve over $K$, of 
positive arithmetic genus ({\it i.e.}, $\dim\mathrm{H}^1(X_K,\cO_{X_K})>0$). When the base $S$ is excellent or $X_K$ is smooth over $K$, $X_K$ admits  
a unique minimal proper regular model {$\Xm$} over $S$ (see \cite{Chin}, Theorem~1.2 or \cite{Liu}, 8.3.45 and 9.3.21). Let $X_\sm$ denote the smooth locus of $\Xm/S$. 
The aim of this section is to prove the next theorem. 
See also Proposition~\ref{nm-conic} for a partial result in higher dimension.

\begin{theorem} \label{pj-nm} Let $S$ be a Dedekind scheme with field of functions
$K$. Let $X_K$ be a proper regular connected curve of positive arithmetic genus over $K$. Assume either $S$ is excellent or $X_K/K$ is smooth. Then $X_\sm$ is the N\'eron model of the smooth locus $X_{K,\sm}$ of $X_K$ over $S$. 
\end{theorem}

We will deduce Theorem~\ref{pj-nm} from the next proposition. 

\begin{proposition}[see also \cite{BLR}, 7.1/6]\label{2nd-app-of-fY-sm} Let $P_K$ be a separated
  connected smooth $K$-scheme of finite type, and let $U_K\subseteq
  P_K$ be a connected smooth closed subscheme of dimension one. Assume that
$P_K$ admits a N\'eron lft-model (resp. N\'eron model) $P$ over $S$. Then
$U_K$ admits a N\'eron lft-model (resp. N\'eron model) over $S$. 
\end{proposition}

\begin{proof} Let $U_{0}$ denote the scheme-theoretic closure of $U_K$ inside $P$. Let $p: U\to U_0$ be the minimal desingularization of $U_0$ (Remark \ref{mini-desing}). We want to prove that the smooth locus
{$U_{\sm}$} of $U\to S$ is the N\'eron lft-model of $U_K$ over $S$. {As} the formation of $U_\sm$ commutes with localization and strict henselization
(\cite{Liu}, Proposition 9.3.28), by Proposition~\ref{nm-gen} we can suppose 
$S=\Spec(R)$ is strictly local ({\it i.e.}, $R$ is a strictly henselian
  discrete valuation ring). Let $Y$ be a smooth scheme over $S$ and
let $f_K\colon Y_K\to U_K$ be a morphism of $K$-schemes. We want to
extend $f_K$ to a morphism of $S$-schemes $Y\to U_\sm$. We can suppose
$Y_s$ is irreducible
(Remark~\ref{univ} (2)). If $Y_K(K)=\emptyset$, then $Y=Y_K$, and 
$f_K$ is a morphism from $Y$ to $U$. 

Suppose $Y_K(K)\ne\emptyset$. 
If $f_K \colon Y_K\rightarrow U_K$ is not dominant, the image  $f_K(Y_K)$ consists of a rational point $q$ of $U_K$. The Zariski closure
$\overline{\{q\}}$ of $\{q\}$ in {$U$} is contained in $U_\sm$ (Lemma~\ref{sm-sect}) and is the image of a section $\sigma\colon S\to U_\sm$. Then $f_K$ extends to $Y\to U_\sm$ as composition of the structure morphism $Y\to S$ and the section $\sigma \colon S\to U_\sm$.  

Now suppose $f_K$ is dominant. 
The morphism $Y_K\to U_K\to P_K$ extends to a dominant morphism 
$Y\to U_0$.  By Corollary~\ref{fY-reg}(2), the latter 
induces a dominant morphism $Y\to U$. Therefore $f_K$ extends to 
$Y\to U$, {hence} to $Y\to U_\sm$ (Proposition~\ref{sm-to-sm}). This shows that $U_\sm$ is the N\'eron lft-model of $U_K$ over $S$. 
\smallskip 

If $P$ is of finite type over $S$, then $U_0$ and $U$ above are 
of finite type over $S$, thus $U_\sm$ is of finite type. 
\end{proof} 

The next proposition is well known. 

\begin{proposition}\label{embedding} Let $k$ be a field, and let $C$ be a projective 
geometrically integral curve over $k$ of arithmetic genus $\geq 1$. Let
$U\subseteq C$ be the smooth locus of $C/k$. 
Then the canonical morphism 
$$U\to \Pic^1_{C/k}, \quad x\mapsto \cO_{C}(x)$$ 
(given by the invertible sheaf $\cO_{C\times_k U}(D)$, where $D$ is the
graph of the inclusion $U\to C$) is a closed immersion. 
\end{proposition}

\begin{proof} 
Let $\Div_{C/k}$ be the scheme of effective Cartier divisors on 
$C$ (see \cite{BLR}, \S 8.2). Let $\Div^1_{C/k}$ be the subscheme
corresponding to effective Cartier divisors of degree $1$. 
Then the canonical morphism 
$U\to \Div_{C/k}^1$, $x\mapsto x$, is an isomorphism
(\cite{KL}, Exercise 9.3.8). 
\if 
Note that the answer given in \cite{KL}, A.9.3.8, seems too complicated
in this situation. 

Indeed, for any algebraically closed field $k'$ containing $k$, an
effective Cartier divisor on $C_{k'}$ of degree $1$ is a smooth
point and thus contained in $U_{k'}$. It is then straightforward
to check that $U$ represents the
functor $\Div_{C/k}^1$, with the universal divisor $D$ on $C\times_k U$. 

More functorially: 
Let $T$ be any $k$-scheme and let $E$ be a relative effective Cartier 
divisor of degree $1$ on $C\times_k T\to T$. Then for any $t\in T$, 
$E_{\bar{t}}$ is an effective Cartier divisor of degree $1$ on 
$C_{\bar{t}}$, so is contained in $U_{\bar{t}}$ and the projection 
$U\times_k T\to T$ induces an isomorphism $E\to T$, 
induces a morphism $E\to U$. ... 
\fi 

Let $f: \Div^1_{C/k}\to \Pic^1_{C/k}$ 
be the restriction of the canonical morphism $\Div_{C/k}\to
\Pic_{C/k}$ (corresponding to $E\mapsto \cO_{C}(E)$). 
It will be enough to show that $f$ is a closed immersion. 
It is known that $f$ can be identified to $\mathbb P(\mathcal F)\to 
\Pic^1_{C/k}$ for some coherent sheaf $\mathcal F$ on 
$\Pic^1_{C/k}$ (\cite{BLR}, Proposition 8.2/7). In particular, 
$f$ is proper and its fibers are projective spaces.  
On the other hand, the map $U\to \Pic^1_{C/k}$ is 
injective because $p_a(C)>0$. So the fiber of $f$ 
at any $y\in \mathrm{Im}(f)$ is  a projective space of dimension $0$
over $k(y)$, hence 
isomorphic to $\Spec(k(y))$. This implies that $f$ is 
a proper (hence closed) immersion. 
\end{proof} 

\noindent{\it Proof of Theorem \ref{pj-nm}:} As $X_{\sm}$ is a
finite type scheme over $S$, it is enough to show $X_\sm$ is the
N\'eron lft-model of {$X_{K,\sm}$}. By 
Corollary~\ref{neron-l2g}, to show our theorem we can suppose $S$ is 
local. Consider the closed immersion $f: X_{K,\sm}\to \Pic^1_{X_K/K}$ 
defined in Proposition~\ref{embedding}. On the other hand, 
$\mathrm{Pic}^1_{X_K/K}$ is a torsor under $J_K:=\mathrm{Pic}^0_{X_K/K}$, and $J_K$ has no subgroup isomorphic to 
$\mathbb G_{a,K}$ or $\mathbb G_{m,K}$ (\cite{Ray2}, Proposition~1.1). Hence $J_K$ has a N\'eron model over $S$ (\cite{BLR}, Theorem 10.2/1)
as well as $\mathrm{Pic}^1_{X_K/K}$ (\cite{BLR}, Corollary 6.5/4). 
By Proposition~\ref{2nd-app-of-fY-sm}, $X_{K, \sm}$ has a N\'eron model $N$
over $S$. Embed $N$ into a proper model, and resolve the 
singularities without modifying the regular locus (which contains $N$). 
Then we get a proper regular model {$N'$} containing $N$ as an open
subset. The identity on $X_K$ extends to a morphism ${N'}\to \Xm$. 
By Corollary \ref{sm-to-sm}, this morphism induces a morphism $N\to
X_\sm$ which is an isomorphism on the generic fiber. Therefore
$X_\sm$ satisfies the N\'eron mapping property. 
\qed 

\begin{remark} Keep the notation of Theorem~\ref{pj-nm}. 
\begin{enumerate}
\item The N\'eron model $X_\sm$ is not necessarily 
faithfully flat over $S$. Indeed, if $\Xm$ has a multiple
fiber above some point $s\in S$, then $(X_\sm)_s=\emptyset$. 
However, the faithful flatness holds if $X_K(K)\ne\emptyset$. 
\item Assume $X_K$ is smooth over $K$ and $X_K(K)\neq \emptyset$, and embed $X_K$ into its
    Jacobian $J_K$ by using some rational point of $X_K$. The
proof of Theorem~\ref{pj-nm} shows that the smooth locus of the minimal desingularization of the scheme-theoretic closure $\overline{X_K}\subset J$ is isomorphic to $X_{\sm}$, where  $J$ denotes the N\'eron model of $J_K$.
Note that in general, even when $\Xm$ is semi-stable, $X_\sm\to J$
is not an immersion, see \cite{Edix}, Proposition 9.5. 
\end{enumerate}
\end{remark}

\begin{remark} \label{nm-ind1} Suppose $S=\Spec(R)$ is local. Let 
$R\to R'$ be an extension of discrete valuation rings of index $1$
({\it i.e.}, ramification index $1$ and separable residue extension),
let $K'=\Frac(R')$. Then under the condition of Theorem~\ref{pj-nm}, 
$X_\sm\times_S \Spec(R')$ is the N\'eron model of $X_{K',\sm}$ over
$R'$. Indeed, the formation of $X_\sm$ {commutes with completion 
(\cite{Liu}, 9.3.28). Applying Proposition~\ref{nm-gen}~(3) to the extension
$\widehat{R}\to \widehat{R'}$, we see that $X_\sm\times_{S} \Spec(\widehat{R'})$
is the $\widehat{R'}$-N\'eron model, hence $X_\sm\times_S \Spec(R')$ 
is the N\'eron $R'$-N\'eron model by Proposition~\ref{nm-gen} (1).} 
\end{remark}

\begin{remark} 
For any smooth connected curve $U_K$ over $K$, if it can be
  embedded as a closed subscheme into a semi-abelian $K$-variety (or
  more generally, into a smooth $K$-group scheme of finite type
  admitting a $S$-N\'eron lft-model),
  Proposition~\ref{2nd-app-of-fY-sm} implies immediately that $U_K$
  admits also an $S$-N\'eron lft-model. For example, this works if
  $U_K$ is proper of positive genus, or if $U_K$
  is affine such that the reduced divisor at infinity of $U_K$ in its regular compactification is separable of degree $>1$ over $K$. But this method does not apply for \emph{all}
  curves: for instance the complement of a rational point in
a proper smooth connected curve can not be embedded as a closed subscheme into any semi-abelian variety over $K$. In the second part of this work (\S \ref{af-nm-s}--\S
\ref{nm-rat-gl}), we propose a different approach which works for any affine curve. Even better, our method allows a very explicit description of the N\'eron lft-model with the help of minimal proper regular model of the regular compactification of the affine curve. 
\end{remark}

\begin{corollary} \label{cov-gr} Let $S$ be a Dedekind scheme with
  field of functions $K$, and let $X_K$ be a connected smooth
 proper curve over $K$ of positive genus. 
Suppose that there exist a proper smooth variety $Y_K$ having 
good reduction ({\it i.e.}, having a proper smooth model $Y/S$) and 
a dominant morphism $f_K\colon Y_K\to X_K$ over $K$. Then $X_K$ 
has good reduction.   
\end{corollary}

\begin{proof} Let $X_\sm$ be the N\'eron model of $X_K$ over $S$. Then $f_K$
extends to $f: Y\to X_\sm$. 
As $Y$ is proper over $S$, the image $f(Y)$ 
is closed in $X_\sm$ and is dense because it contains $X_K$. So $f(Y)=X_\sm$.
Therefore $X_\sm$ is proper over $S$, and $X_K$ has good reduction. 
\end{proof}

\begin{remark} One can give a direct proof of Corollary~\ref{cov-gr}
  if the smooth $S$-scheme $Y/S$ is assumed to be
  \emph{projective}. Indeed, it is enough to show $X_\sm=\Xm$. So 
one can suppose $S$ is local. 
Using Bertini-type result, we may find a
  smooth closed subscheme of $Y/S$ such that its generic fiber dominates
  again $X_K$. Then we can repeat this argument to lower the relative
  dimension of 
the scheme $Y/S$ until we find a smooth relative curve over $S$ whose
generic fiber still dominates $X_K$. Finally we only need to apply
\cite{LL}, Corollary 4.10 to conclude. 
\end{remark}

\begin{remark} Corollary \ref{cov-gr} does not hold in general if $g(X_K)=0$. 
Let us consider the following example. Let 
$R$ be a henselian discrete valuation ring with finite residue 
field $k$. 
Let $k'/k$ be a quadratic extension and let $T^2+aT+b\in R[T]$ be a lifting 
of the minimal polynomial of a generator of $k'/k$. 
Consider the scheme 
$$X=\Proj R[x,y,z]/(x^2+axy+by^2+\pi z^2)$$
where $\pi$ is a uniformizing element of $R$. 
Let $R'=R[T]/(T^2+aT+b)$. This is a finite \'etale
extension of $R$.  Let $K=\Frac(R), K'=\Frac(R')$. Then 
$X_{K'}\cong 
\mathbb P^1_{K'}$. So $X_{K'}$ has good reduction over $R'$, hence
over $R$ when it is viewed as a $K$-scheme (a smooth projective 
model over $R'$ is also smooth projective over  $R$). 

We have a surjective $K$-morphism $X_{K'}\to X_K$. However, 
$X_K$ does not have good reduction over $S$. Indeed, 
suppose $X_K$ has a proper smooth model $P\to S$. Then $P_k$ 
is a smooth conic over a finite field, hence has a rational point. 
As $R$ is henselian, then $X_K=P_K$ has a rational point. The latter 
then specializes to the (unique) rational point of $X_k$. But
this is a singular point of $X_k$. As $X$ is regular, we 
have a contradiction by Lemma~\ref{sm-sect}. 
\end{remark}

\begin{corollary} Let $X_K$ be a proper smooth connected curve of 
positive genus over $K$. Suppose that for some proper smooth connected
variety $Y_K$ over $K$, the product $X_K\times_K Y_K$ has good
reduction over $S$. Then $X_K$ has good reduction over $S$.   
\end{corollary}

\begin{proof} Apply Corollary \ref{cov-gr} to the projection 
$X_K\times_K Y_K\to X_K$.
\end{proof}

Let $f_K\colon Y_K\to X_K$ be a finite morphism of {proper} smooth and connected curves over $K$. Suppose $g(X_K)\ge 1$. Let $Y_{\mathrm{min}}, \Xm$ be the respective minimal proper regular
models of $Y_K, X_K$ over $S$. In general, $f_K$ does not extend to a morphism $Y_{\mathrm{min}}\to \Xm$ (\cite{LL}, Remark 4.5). However, Theorem~\ref{pj-nm} immediately implies the next corollary, answering positively a question raised by A. Pirutka. 

\begin{corollary} Let $f_K: Y_K\to X_K$ be a finite morphism of 
{proper} smooth and connected curves over $K$, with $g(X_K)\ge 1$.  Let $Y_\sm$ (resp. $X_\sm$) be the smooth
locus of the minimal proper regular model of $Y_K$ (resp. $X_K$) over $S$. 
Then $f_K$ extends to a morphism $f : Y_\sm \to X_\sm$. 
\end{corollary}

For the sake of completeness, let us consider the N\'eron 
lft-models of curves of genus $0$. 

\begin{proposition}\label{nm-conic} Let $X_K$ be a smooth projective
conic over $K$. 
\begin{enumerate}[{\rm (1)}] 
\item If $X_K=\mathbb P^1_K$, then $X_K$ does not have N\'eron lft-model
over $S$. 
\item In general, $X_K$ has a N\'eron lft-model over $S$ if and only
if $S$ is semi-local and if $X_K(K_s^{\sh})=\emptyset$  
for any closed point $s\in S$, where $K_s^{\sh}$ denotes the fraction 
field of the strict henselization of $\cO_{S,s}$. In this case, $X_K$ is its
own N\'eron model over $S$. 
\end{enumerate}
\end{proposition}

\begin{proof} (1) We will argue by contradiction:
assume that $\mathbb P_K^1$ admits a N\'eron lft-model $P$ over $S$.
By the N\'eron mapping property, there exists a 
  morphism of $S$-schemes $f\colon \mathbb{P}_{S}^{1}\rightarrow P$
  extending the canonical identification between the generic fibers. 
We claim that $f$ is an isomorphism. As $\mathbb P_S^1$ is proper, and
$P/S$ is separated, the image $f(\mathbb P_S^1)$ is a closed subset of $P$.
Its closed fiber has dimension $1$ by Chevalley's semi-continuity 
theorem (\cite{EGA}, IV.13.1.1). Therefore $\mathbb P^1_{k(s)}\to P_s$
is quasi-finite.  By Zariski's Main Theorem, $f$ is an open
immersion. But $f$ is proper and $P$ is irreducible, thus 
$f : \mathbb{P}_{S}^{1}\to P$ is an isomorphism. 

On the other hand, there are many endomorphisms of
$\mathbb{P}_{K}^{1}$ that 
can not extend to an endomorphism of $\mathbb{P}_{S}^{1}$, hence
$\mathbb{P}_{S}^{1}$ is not the N\'eron lft-model of $\mathbb
P_K^1$. Contradiction.

(2) Assume first that $X_K$ admits a N\'eron lft-model over $S$. 
By Proposition~\ref{nm-gen}~(2), $X_{K_s^{\sh}}$ admits a N\'eron lft-model over 
$\mathrm{Spec}(\cO_{S,s}^{\sh})$. As a result, 
$X_{K_s^{\sh}}\not\cong \mathbb{P}_{K_s^{\sh}}^1$. In other words, 
$X_K(K_s^{\sh})=\emptyset$ for all closed points $s\in S$. On the other hand, this condition implies that $S$ is semi-local.
Conversely, assume $X_K(K_s^{\sh})=\emptyset$ for any
closed point $s\in S$. As $S$ is semi-local, $X_K$ is of finite type over
$S$ and is its own $S$-N\'eron model. 
\end{proof}

\noindent {\bf Higher dimension.} \ 
Let $V$ be an algebraic variety 
over an algebraically closed field $k$. We say that 
\emph{$V$ contains a rational curve} if $V$ has a subscheme isomorphic to an open 
dense subscheme of $\mathbb P^1_k$. 
It is easy to see that if $V$ does not contain any rational curve,  then for any algebraically 
closed field extension $K/k$, $V_K$ does not contain any rational curve.

\begin{proposition} \label{hd-nm} Let $S$ be a Dedekind scheme with field of functions $K$. Let $X_K$ be a smooth proper
  algebraic variety over $K$. Suppose $X_K$ has a proper regular 
model $X$ over $S$ such that no  geometric fiber $X_{\bar{s}}$, 
$s\in S$, contains a rational curve. Then the 
smooth locus $X_\sm$ of $X$ is the N\'eron model
of $X_K$. 
\end{proposition}

\begin{proof}
Let $Y$ be an irreducible smooth scheme of finite type over $S$ and let $f_K \colon Y_K\to X_K$ be a morphism of $K$-schemes. Consider the $Y$-scheme $Z:=X\times_S Y\to Y$. Its geometric fibers are 
$Z_{\bar{y}}=X_{s}\otimes_{k(s)} \overline{k(y)}$ and they do not 
contain any rational curve. The morphism $f_K$ induces a section $Y_K\to Z_K$,
$y\mapsto (f_K(y),y)$,  which extends to a section $Y\to Z$ by 
\cite{GLL2}, Proposition 6.2. Composing this section with the
projection $Z\to X$ gives a morphism $f\colon Y\to X$ extending $f_K$. By Proposition~\ref{sm-to-sm}, $f(Y)\subseteq X_\sm$. This proves that $X_\sm$ is the N\'eron model of $X_K$ over $S$. 
\end{proof}

\begin{remark} If $A_K$  is an abelian variety having good reduction over $S$,
then we  recover the well-known fact that a proper smooth model $A$
of $A_K$ is the N\'eron model of $A_K$ (\cite{BLR}, Proposition 1.2/8). 
\end{remark}

\end{section} 

\begin{section}{N\'eron lft-models of open curves in the local case} \label{af-nm-s}

Let $X_K$ be a separated smooth connected curve having a N\'eron
lft-model $X$ over $S$ ({\it e.g.} if $X_K$ is proper of positive genus,
see Theorem \ref{pj-nm}). Let $U_K$ be an open dense subscheme of $X_K$. 
A natural question is whether $U_K$ has a N\'eron lft-model $U$ 
and, if it exists, how it is related to $X$. In this section, 
we restrict ourselves to the case where $S=\mathrm{Spec}( R)$ is \emph{local}. 
Then we will show that the answer is positive under mild hypothesis 
and we describe explicitly the construction of $U$. 

Let $R^{\sh}$ denote the strict henselization of $R$, $\widehat{R^{\sh}}$ 
the completion of $R^{\sh}$, $K^{\sh}=\mathrm{Frac}(R^{\sh})$ and 
$\widehat{K^{\sh}}=\Frac(\widehat{R^{\sh}})$. 

\subsection{Main statement} 

\begin{theorem} \label{nm-open-st} 
Let $X_K$ be a separated smooth
connected curve over $K$ and let $U_K$ be a dense open subscheme of 
$X_K$. Suppose that $X_K$ has a smooth model $X$ over $S$ such that 
$X_{\widehat{R^{\sh}}}$ is the N\'eron lft-model of
$X_{\widehat{K^{\sh}}}$. Then $U_K$ has a N\'eron lft-model $U$ over $S$. Moreover, 
denoting by $\Delta_K=X_K\setminus U_K$ the boundary of $U_K$ in
$X_K$, 
$U$ satisfies the following properties. 
\begin{enumerate}[{\rm (1)}] 
\item The scheme $U$ is of finite type over $S$ if and only if 
$X$ is of finite type and if 
$$\Delta_K\cap X_K({\widehat{K^{\sh}}})=\emptyset.$$ 
{If $S$ is excellent, the latter condition is also equivalent to 
\[
\Delta_K\cap X_K(K^{\sh})=\emptyset;
\]} 
\item Let $\Delta$ be the
  Zariski closure of $\Delta_K$ in $X$. Then the identity on the
generic fiber $U_K$ extends to an open immersion 
$X\setminus \Delta\to U$, and the open immersion 
$U_K\to X_K$ extends to a morphism $U\to X$. 
\item Let $s$ be the closed point of $S$ and let $k(s)^{\mathrm{sep}}$ be a separable closure of {the residue field} $k(s)$ of $s$. 
Then $U=X\setminus \Delta$ if and only if 
$\Delta\cap X_s(k(s)^{\mathrm{sep}})=\emptyset$. 
  \end{enumerate}
\end{theorem}

\begin{remark}\label{complement-nm-open-st} 
In the statement of Theorem~\ref{nm-open-st}, the $S$-model $X$ is
necessarily the N\'eron lft-model of $X_K$ over $S$
(Proposition~\ref{nm-gen}(1)), but the requirement
in Theorem~\ref{nm-open-st} is slightly stronger than this, 
except when $S$ is excellent (Proposition~\ref{nm-gen}~(3)).
\end{remark}

\begin{subsection}{Construction of \texorpdfstring{$U$}.} \label{construct-U} 

Let $k=k(s)$. First we construct a (possibly infinite) sequence
of blow-ups of $X$, then define $U$ as a suitable open subset of the
resulting scheme. Note that the following construction 
can be done for any smooth $S$-scheme $X$ such that $U_K$ is a dense
open subset of $X_K$, and Lemmas~\ref{dilat} and \ref{af-pre} hold
just under these assumptions. 

Put $X_0:=X$ and  $\Delta_0:=\Delta$. 
\begin{enumerate}[{(i)}] 
\item Let $\Delta_0'=(\Delta_0)_s\cap \sm(X_0)_s(k^{\mathrm{sep}})$, {\it i.e.}, the subset of points of $(\Delta_0)_s \cap \sm(X_0)_s$ with separable residue field over $k$. If $\Delta'_0=\emptyset$, we stop. 
\item Otherwise, let $X_1\to X_{{0}}$ be the blowing-up of $X_{{0}}$ along 
$\Delta'_0 $ (endowed with the reduced structure). 
Let $\Delta_1$ be the Zariski closure 
of $\Delta_K$ in $X_1$ and
let $\Delta_1'=\Delta_1\cap \sm(X_1)_s(k^{\mathrm{sep}})$. If
$\Delta'_1=\emptyset$, we stop. 
\item Otherwise {blow up} 
$X_2\to X_1$ along $\Delta'_1$ (reduced) and start again with
the Zariski closure $\Delta_2$ of $\Delta_K$ in $X_2$. We construct in this way
a (possibly infinite) sequence of models {locally} of finite type 
$X_n$ of $X_K$, with $\Delta_n\subset X_n$ the Zariski
closure of $\Delta_K$ in $X_n$, and $X_{n+1}\to X_n$ is the
blowing-up of $\Delta_n':=\Delta_n\cap \sm(X_n)_s(k^{\mathrm{sep}})$. 
\end{enumerate}

Let $U_n=\sm(X_n)\setminus \Delta_n$. 
The identity map on $U_K$ extends to an open immersion $U_n\to
U_{n+1}$. More precisely $U_n$ is $U_{n+1}$ minus the 
the exceptional divisor of $X_{n+1}\to X_n$. 
Let $U=\bigcup_{n\ge 0} U_n$. This is a smooth, separated scheme locally of 
finite type over $S$, with generic fiber isomorphic to $U_K$. 
We will show that $U$ is the N\'eron lft-model of $U_K$ over $S$. 
By construction, we have canonical morphisms 
\begin{equation} \label{X-U} 
X\setminus \Delta \hookrightarrow  U\to X,
\end{equation}
and the second morphism has image in 
$(X\setminus  \Delta)\cup {(\Delta_0')_{s}}$.
\end{subsection}

{Note that the formation of $U$ commutes with 
any flat extension of discrete valuation rings $R\to R'$ because
taking Zariski closure, 
blowing-up and taking the smooth locus 
are all compatible with such an extension. 
To prove $U$ is the N\'eron lft-model, we can replace $R$ by $\widehat{R^{\sh}}$ and suppose $R$ is strictly henselian and excellent (Proposition~\ref{nm-gen}). After this reduction, $X$ is the N\'eron lft-model of $X_K$ in the situation of Theorem~\ref{nm-open-st}.}
But in general the curve $U_K$ might be no longer connected
(the curve $X_K$ and its N\'eron lft-model decompose accordingly). In
this situation, if we can deal with each connected component of $U_K$
(which is an open subscheme of a connected component of $X_K$), the
general case follows (Remark~\ref{univ}~(3)). Hence, we may assume that $U_K$ is still connected.

\begin{subsection}{Dilatation} \label{dilat-df} (\cite{BLR}, \S 3.2). 
Let $R$ be a discrete valuation ring. Let 
$X$ be any flat $R$-scheme of finite type. Let $E$ be a closed subscheme 
of the special fiber $X_s$ defined by a sheaf of ideals $\mathcal
I\subset \cO_X$. Recall that the \emph{dilatation of $E$ on $X$} 
is obtained by blowing-up $u: \widetilde{X}\to X$ along $E$, and 
then taking the open subset of $\widetilde{X}$ where $u^*\mathcal I$ is 
generated by a uniformizing element of $R$. Denote by $X'$ the dilatation of $E$. It satisfies the following universal 
property: for any flat $R$-scheme $Z$, a morphism $Z\to X$ 
factors through $X'\to X$ if and only if $Z_s\to X_s$ factors 
through $E\to X_s$ as morphism.  

{If $X/R$ is smooth and if the center $E$ is smooth over $k$, by a local computation, one can show that the dilatation $X'$ of $E$ on $X$ is smooth over $R$, and is equal to} the
complement in $\widetilde{X}$ of the strict transform of $X_s$. 
\end{subsection}

\begin{subsection}{Some technical lemmas} 

\begin{lemma} \label{dilat} Keep the notation of \S \ref{construct-U}
and assume $R$ is strictly hense\-lian and excellent.  
Suppose that the sequence of blowing-ups 
$X_{n+1}\to X_n$ is infinite. Let $p_{1}, p_{2}\in X_K$ be closed points such that for all 
$n\in \mathbb N$, they specialize to a same point $x_n\in \Delta_n'$. Then 
$p_{i}\in X_K(K)$ and $p_{1}=p_{2}$. 
\end{lemma}

\begin{proof}
Let us first prove $p_{i}\in X_K(K)$. Let $p\in \{p_{1}, p_{2}\}$. 
Let $X'_{n+1}\subseteq X_{n+1}\to X_n$ be the {dilatation} of $\Delta_n'$ on
$X_n$ and let $P_n=\{p, x_n\}$ be the reduced Zariski closure of $p$ in 
$X_n$. By the construction of $X_n$, $x_{n+1}$ maps to $x_n\in \Delta_n'$
and it is a smooth point of $X_{n+1}$. 
So $x_{n+1}\in X'_{n+1}$ and 
for all $n\ge 0$, we have a commutative diagram
$$
\xymatrix
{ P_{n+1} \ar[r] \ar[d] & X'_{n+1}\subset X_{n+1} \ar[d]\\ 
P_n \ar[r] & X_n \\ 
} 
$$
The first vertical arrow is birational and finite. Moreover, 
$P_{n+1}$ is nothing but the strict transform of $P_n$ in 
$X_{n+1}$. By the embedded resolution of singularities 
(see {\it e.g.}, \cite{Liu}, 9.2.26, recall that $S$ is excellent), there exists $m\ge 1$
such that $P_m$ is a regular scheme. Then 
$P_{m+1}\to P_m$ is an isomorphism and we have a factorization 
$$
\xymatrix
{ {} & X'_{m+1} \ar[d]\\ 
P_m \ar[r]^{\sigma_m} \ar[ur]& X_m. \\ 
} 
$$
By the universal property of the dilatation (\S \ref{dilat-df}), 
the closed immersion $\sigma_{m,{s}} \colon P_{m,{s}}\to X_{m,{s}}$ 
factors through $\Spec k(x_m)\subset X_{m,{s}}$. 
{Hence $P_{m,s}\to \Spec k(x_m)=\Spec k$ is an isomorphism.} 
Therefore $P_m\to S$ is an isomorphism and we have $K(p)=K$. 

If $p_{1}\ne p_{2}$, by the embedded 
resolution of singularities of the Zariski closure of 
$\{ p_{1}, p_{2}\}$ in $X_0$, the Zariski closure becomes  
a disjoint union of sections in some $X_n$. 
Contradiction with the hypothesis in the lemma. 
\end{proof}

\begin{lemma} \label{af-pre} Keep the notation of \S \ref{construct-U}
and suppose $R$ is strictly hense\-lian and excellent.  
\begin{enumerate}[{\rm (a)}]
\item If $(X_n)_n$ is an infinite sequence, and if 
$(x_n)_n$ is such that $x_n\in \Delta'_n$ and $x_{n+1}\mapsto x_n$
by $X_{n+1}\to X_n$ for all $n\ge 0$, then there exists $p\in \Delta_K$ such that 
$x_n\in P_n$ (Zariski closure of $p$ in $X_n$) for all $n\ge 0$. 
\item Let $Y$ be a flat integral $S$-scheme of finite type with irreducible closed fiber $Y_s$. 
Let $f_K : Y_K \to X_K$ be a morphism. Suppose that for all 
$n\ge 1$, $f_K$ extends to $f_n : Y\to X_n$ and that 
$f_n(Y_s)\subseteq \Delta_n'$. Then 
$f_K(Y_K)\cap \Delta_K\ne\emptyset$. 
\item \label{af-ft} The scheme $U$ is of finite type over $S$ if and only if 
$X$ is of finite type and $\Delta_K\cap X_K(K)=\emptyset$. 
\end{enumerate}
\end{lemma}

\begin{proof} (a) Let $F_n\subseteq \Delta_K$ be the set of points
specializing to $x_n$. Then $(F_n)_n$ is a decreasing sequence of 
non-empty finite sets, so $\cap_n F_n\ne\emptyset$ and any 
point $p$ in the intersection satisfies the required property. 

(b) As $Y_s$ is irreducible, $f_n(Y_s)$ is reduced to one point $x_n$
and $x_{n+1}\mapsto x_n$ by $X_{n+1}\to X_n$.  
Let $p\in \Delta_K$ be such that $x_n\in P_n$ for all $n\ge 0$. 
Let $q\in Y_K$ be a lifting of some {closed}  
point of $Y_s$. Then $f_K(q)$ and $p$ are closed points of $X_K$ having the same specialization $x_n\in \Delta'_n$ for all $n$. By Lemma~\ref{dilat}, $p=f_K(q)\in \Delta_K\cap f_K(Y_K)$. 

(c) Suppose there exists $p\in \Delta_K\cap X_K(K)$. For any $n\ge 0$ such
that $X_n$ is constructed, $p$ specializes to a point of 
$\Delta_n\cap \sm(X_n)_s(k)$, so $X_{n+1}$ exists in the 
construction of \S~\ref{construct-U}. As $X_{n+1}\to X_n$ consists in
blowing-up some smooth points, $U_{n+1}$ contains a dense open
subset of the exceptional locus of $X_{n+1}\to X_n$. So 
$U_n\subsetneq U_{n+1}$ and $U$ is not of finite type. 
{On the other hand, if $X$ is not of finite type,
then $U_1$, thus $U$, is not of finite type.}

Conversely, if $\Delta_K$ does not contain rational point of 
$X_K$, there exists $m\ge 1$ such that no point of $\Delta_K$ 
specializes to a point of $\sm(X_m)_s(k)$ (Lemma~\ref{dilat}), and the 
construction of \S~\ref{construct-U} stops at this step. Thus
$U=U_m$ is of finite type over $S$ if $X$ is of finite type. 
\end{proof}

\begin{lemma} \label{fd-surj} {Let $S$ be a locally noetherian scheme, and }
let $f: Z\to T$ be a morphism of locally noetherian flat $S$-schemes. 
Let $z_0\in Z_s$ and $t_0=f(z_0)$. Suppose 
$Z_s\to T_s$ is flat at $z_0$. Then the canonical morphism 
$$f_{z_0} : \Spec \cO_{Z,z_0} \to \Spec \cO_{T, t_0}$$ 
is surjective. In particular, for any $P\in T$ such that 
$t_0\in \overline{\{ P\}}$, there exists $Q\in f^{-1}(P)$ such that
    $z_0\in \overline{\{ Q\}}$. 
  \end{lemma}

 \begin{proof}
By the fiberwise flatness criterion (\cite{EGA}, IV.11.3.10.1), 
$f$ is flat at $z_0$ and $f_{z_0}$ is a flat morphism of local
schemes, hence faithfully flat. In particular, $f_{z_0}$ is
surjective.
The last assertion results from the usual interpretation of  
the images of $\Spec\cO_{T,t_0}$ and $\Spec\cO_{Z,z_0}$ in $T$ and
$Z$ respectively.
  \end{proof}
\end{subsection}

\begin{subsection}{Proof of Theorem~\ref{nm-open-st}}\label{proof-nm-open-st} 
Now we prove that $U$ is the N\'eron lft-model of $U_K$ over $S$. 
As noticed previously in \S \ref{construct-U}, we can suppose 
$S$ is strictly local {and excellent}. In particular $k^{\mathrm{sep}}=k$. 

Let $Y$ be a smooth scheme over $S$ and let $f_K : Y_K \to U_K$ be a 
morphism of $K$-schemes. We want to extend $f_K$ to a morphism $Y\to U$. 
We can suppose $Y_s$ is irreducible (Remark~\ref{univ}). First $f_K$ 
extends to a morphism $f_0 \colon Y\to X$ by the hypothesis on $X$. 
Consider the sequence $(X_n)_n$ constructed in \S \ref{construct-U}.

(A) Suppose that for some $m\ge 0$, $f$ factors through 
$f_m \colon  Y\to X_m$ and that $f_m(Y_s)\not\subseteq
\Delta'_m$. Let us show $f_m(Y)\subseteq U_m$ or, equivalently,
that $f_m(Y_s)\cap \Delta_m=\emptyset$ because $f_m(Y)\subseteq
\sm(X_m)$. If this is true, then $f_K$ 
extends to $f_m \colon Y\to U_m\subseteq U$ and we are done. 
We distinguish two cases: 
\smallskip 

{\bf Case 1: \it $f_m(Y_s)=\{ x_m\}$ is a singleton}. 
By hypothesis, $x_m\notin \Delta_m'$.  
As $Y_s$ is smooth and $Y_s\to \Spec k$ factors through
$Y_s\to \Spec k(x_m)$, we have $x_m\in \sm(X_m)_s(k)$. 
So $x_m\notin \Delta_m$ and $f_m(Y_s)\subseteq U_m$. 
\smallskip 

{\bf Case 2: \it $f_m(Y_s)$ is not a singleton}. If there exists 
$x_m\in f_m(Y_s)\cap \Delta_m$, let $y_m\in f_m^{-1}(x_m)$
and let $p\in \Delta_K$ specializing to $x_m$.
As $\sm(X_m)_s$ is a smooth curve and $(f_m)_s$ is dominant, 
$(f_m)_s$ is flat at $y_m$. By Lemma~\ref{fd-surj} applied to
$Z=Y$ and $T=\sm(X_m)$, we find a closed point $q\in Y_K$
such that $f_K(q)=p$. This contradicts the hypothesis
that $f_K(Y_K)\subseteq U_K$. 
\smallskip 

(B) Now we show that the condition in (A) is satisfied for 
some $m\ge 0$. We start with $f_0 \colon Y\to X_0$. If 
$f_0(Y_s)\subseteq \Delta'_0$, then $f_0(Y_s)=\{ x_0\}$. 
As in Corollary~\ref{fY-reg}(1), 
$f_0$ factors through $f_1 : Y\to X_1$. If 
$f_1(Y_s)\not\subseteq \Delta_1'$, we are done. Otherwise,
we have a $f_2\colon  Y\to X_2$. Repeating
this construction, we see that if (A) is never
satisfied, then for all $n\ge 0$, $f_K$ extends to 
$f_n \colon  Y\to X_n$ with $f_n(Y_s)\subseteq \Delta_n'$. 
This is impossible by Lemma~\ref{af-pre}(b). Hence $U$ is 
the N\'eron lft-model of $U_K$.
\smallskip

It remains to prove the various properties of $U$. 
We first remark that if $S$ is excellent, then $K^{\sh}$ is algebraically
closed in its completion $\widehat{K^{\sh}}$,
hence the two conditions in Part~(1) are indeed equivalent.
Part (1) is a direct consequence of Lemma~\ref{af-pre}(\ref{af-ft}). 
Parts (2) and (3) follow from the construction in \ref{construct-U}: 
$X\setminus \Delta_0=U_0$ is open in $U$ and its generic fiber is 
$X_K\setminus \Delta_K=U_K$; and  $U=X\setminus \Delta_0$ is equivalent 
to $\Delta'_0=\emptyset$. \qed
\end{subsection}

\medskip 

{As an application of Theorem~\ref{nm-open-st}, we have the following result which will be used in Proposition~\ref{prop-rat-nm2} in the next section}. 

\begin{proposition} \label{nm-P} Let $S$ be local. 
Let $P$ be a regular, proper semi-stable model of $\mathbb P^1_K$
over $S$. Let $\Gamma_1, \cdots, \Gamma_n$ be disjoint sections,
$n\ge 2$, 
in the smooth locus $\sm(P/S)$ such that $\Gamma:=\cup_i \Gamma_i$ 
meets every exceptional divisor of $P/S$ when $P_s$ is not irreducible. 
Let $V_K=\mathbb P^1_K\setminus \Gamma_K$. Let $V$ be the 
$S$-model over $S$   
obtained by the process described in \S \ref{construct-U} starting with
$X:=\sm(P/S)$. Then $V$ is the $S$-N\'eron lft-model of $V_K$ over $S$. 
Moreover, the identity $V_K\to V_K$ and the
inclusion $V_K\to P_K$ extend to 
$$\sm(P/S)\setminus \Gamma\to V\to \sm(P/S)$$ 
and the first morphism is an open immersion. 
\end{proposition}

\begin{proof} {As the formation of $V$ commutes with completion and strict henselization, by Proposition~\ref{nm-gen}~(1)}, we can suppose
$S$ is strictly local and excellent. We prove the result 
by induction on the number of irreducible components of $P_s$. 

First suppose $P_s$ is irreducible. Start with the case $n=2$. 
Then $P\cong \mathbb P^1_S$. 
One can see easily that $V$, which is isomorphic to the N\'eron 
lft-model of $\mathbb G_{m,K}$, 
is obtained by the process described in \S \ref{construct-U} with
$X:=P$. See \cite{BLR}, 10.1. 

If $P_s$ is irreducible and $n\ge 3$, we consider $U_K=\mathbb P_K\setminus 
\{ (\Gamma_1)_K, (\Gamma_2)_K\}$ with its N\'eron lft-model $U$. Then $V$ 
can be obtained 
by the process of \S \ref{construct-U} starting with $X:=U$. As $P\setminus (\Gamma_1\cup \Gamma_2)$ is open in $U$ by the
above discussions, and the Zariski closure of $(\Gamma_i)_K$ in $P$
is $\Gamma_i\subset P\setminus (\Gamma_1\cup \Gamma_2)$ when $i\ge 3$, 
Theorem~\ref{nm-open-st} says that $V$ is the $S$-N\'eron lft-model of $V_K$, 
$P\setminus \Gamma$ is 
open in $V$ and the latter maps to $U$, hence to $P$. 

Now suppose $P_s$ has more than one component. 
Let $E$ be an exceptional divisor in $P$. Up 
to renumbering, we can suppose $\Gamma_1, \dots, \Gamma_r$, $r\le n-1$, 
are exactly the sections of $P$ among the $\Gamma_i$'s not meeting $E$. 
Let $\pi: P\to Q$ be the contraction of $E$ et let $q=\pi(E)\in Q_s$. 
Consider $U_K=\mathbb P^1_K\setminus \{(\Gamma_1)_K, \dots, (\Gamma_r)_K, 
(\Gamma_{r+1})_K\}$. 
Then $Q$ is regular, proper and semi-stable and, if we still denote by $\Gamma_{i}$ the Zariski closure of $(\Gamma_{i})_K$ 
in $Q$,  
$\Gamma_1, \dots, \Gamma_r, \Gamma_{r+1}$ correspond to $r+1$ disjoint sections
of $Q/S$ whose union meets every exceptional divisor of $Q/S$ when $Q_s$ is not 
irreducible. By the induction hypothesis, the N\'eron
lft-model $U$ {of $U_K$} is obtained by the process of \S \ref{construct-U} 
starting from $X:=\sm(Q/S)$. As $(\Gamma_{r+1})_s=\{q\}$ and 
$\pi\colon P\rightarrow Q$ is the blow-up of $Q$ along $\{q\}$, 
by the explicit construction of $U$, $\sm(P)\setminus (\cup_{i\le r+1}\Gamma_i)$ 
is open in $U$. For any $i\ge r+2$, the point of 
$(\Gamma_i)_K$ specializes to {a point of $(\sm(P)\setminus (\cup_{i\le r+1}\Gamma_i))_s$}. 
Then Theorem~\ref{nm-open-st} tells us that $V_K$ {admits an $S$-N\'eron lft-model $V$, and the latter} is obtained from 
$U$ by blowing-up {the closed points $\cup_{i\ge r+2}(\Gamma_i)_s$ contained in the open subset 
$\sm(P)\setminus (\cup_{i\le r+1}\Gamma_i)\subset U$}, taking
the smooth locus and
start again etc. In particular  $\sm(P/S)\setminus \Gamma$ is an
open subscheme of $V$ and the latter maps to $\sm(P/S)$. 
\end{proof}

\end{section}

\begin{section}{N\'eron models of open subsets of a smooth conic}
\label{rat-nm} 

In this section, we suppose $S=\Spec(R)$ is local and excellent. 
We prove the existence of the N\'eron model for affine open subsets of a smooth projective conic, whose complement is non-empty and consists of ramified points. 

\begin{proposition} \label{prop-rat-nm2}Let $S$ be an excellent local Dedekind 
scheme with field of functions $K$. Let $C_K$ be a projective 
{smooth conic over $K$}. Let $\Delta_K$ be a {non-empty}
finite closed subset of $C_K$ {(endowed with the reduced structure)}
such that $\Delta_K(K^{\sh})=\emptyset$. Then 
$U_K:=C_K\setminus\Delta_K$ admits a N\'eron model over $S$.
\end{proposition}

\begin{proof} 
Assume first $C_K(K)\neq \emptyset$, or equivalently, 
$C_K\cong \mathbb{P}_{K}^{1}$.  Consider a smooth proper model 
isomorphic to $\mathbb P^1_S$ of $\mathbb P^1_K$. 
After finitely blowing-ups along smooth separable
points of the special fiber of the latter, the construction of 
\S \ref{construct-U} gives us  a regular proper semi-stable model 
$P$ of $C_K\cong\mathbb P^1_K$ such that the intersection 
$\Delta\cap \sm(P/S)_s(k(s)^{\mathrm{sep}})$ is empty. By successively blowing-down 
exceptional divisors of $P$ which do not meet the Zariski
closure $\Delta$ of $\Delta_K$, we can suppose that:
\begin{enumerate}
\item $\Delta$ {meets every exceptional divisor of} $P/S$ {if $P_s$ is not irreducible} and
\item $\Delta$ meets $P_s$ only at singular points or smooth inseparable points. 
\end{enumerate}
\noindent{\bf Claim:} {\sl under the above conditions, 
$U:=\sm(P/S)\setminus \Delta$ is the N\'eron model of $U_K$.}

To prove the claim we can suppose $S$ is strictly local. In
particular, each (reduced) irreducible component of $P_s$ is
isomorphic to $\mathbb P_{k}^{1}$. 
Let $Y$ be a smooth $S$-scheme with connected fibers and let $f_K: Y_K\to U_K$ be a morphism
of $K$-schemes. Let $\Gamma_{1}, \dots, \Gamma_n \subset U(S)$ be 
disjoint sections such that $n\ge 2$, 
$\Gamma:=\cup_i \Gamma_i$ is ample 
in $P$ and such that $f_K(Y_K)\not\subset \Gamma_K$. 

Set $H=\overline{f_{K}^{-1}(\Gamma_K)}\subset Y$. This 
is a closed subset of $Y$, empty or of codimension $1$. 
Let $Y':=Y\setminus H$. We claim that the restriction 
$f_K|_{Y_K'} : Y'_K\to U_K\setminus \Gamma_K$ extends to a 
morphism $Y' \rightarrow U$. Indeed,  
let $V$ be the N\'eron lft-model of $V_K:=P_K\setminus \Gamma_K$.  
Then $f_K$ induces a morphism $f': Y'\to V$. 
By Proposition~\ref{nm-P}, we have a morphism $V\to \sm(P/S)\setminus \Gamma=U$. 
Consequently, $f'$ extends to a morphism $f'' : Y''\to U$ where $Y''=Y'\cup Y_K$. {As $Y_s$ is connected, the image $f''_s\colon Y_s''\to U_s$ is contained in some connected component. Let $U'$ denote the union of the latter connected component of $U_s$ with $U_K$. Then $f''$ factors through $U'\subset U$.}
 
On the other hand, {the scheme $U'$ can be obtained by first blowing-down successively all the irreducible components of $P_s$ other than the one containing $U_s '$, then removing from the resulting proper smooth model of $C_K\cong\mathbb P_K^1$ the closure of $\Delta_K$. In particular,} the scheme $U'$ is affine.  Thus $f''$ 
extends to a morphism $Y\rightarrow U'\subset U$ since $Y\setminus Y''$ is a closed subset of 
codimension $\geq 2$ of the normal scheme $Y$ (\cite{BLR}, 4.4/2). Therefore $U$ satisfies 
N\'eron mapping property, as desired. 

For the general case, if $X_K(K^{\sh})=\emptyset$, then $X_K$ is the $S$-N\'eron model of itself. Otherwise,  
there exists a finite {unramified} extension $K'$ of $K$ 
such that $C_K(K')\neq \emptyset$. Since $K'/K$ is {unramified}, the complement of $U_{K'}$ in its smooth compactification $C_{K'}$ consists of closed points which {are} still {ramified} of degree $>1$ over $K'$. Therefore, $U_{K'}$ admits a N\'eron model $U'$ over the semi-local ring $S'$, the normalization of $S$ in $K'$. 
Hence $U_K$ admits also a $S$-N\'eron model by Proposition~\ref{neron-weil} {and Proposition~\ref{quasi-pj}}. 
\end{proof}

\begin{remark}\label{comp-to-prop-rat-nm2} Keep the notation of Proposition~\ref{prop-rat-nm2}, and let $U$ be the $S$-N\'eron model of $U_K$. Assume $C_K\cong \mathbb P_K^1$ with a proper smooth $S$-model $C$
such that the Zariski closure $\Delta$ of $\Delta_K$ (with the reduced structure) is regular. Then the canonical morphism obtained from N\'eron mapping property $C\setminus \Delta\rightarrow U$ is an open immersion. Indeed, as $\Delta$ is regular, after blowing up all separable closed points of $\Delta$, we obtain a proper semi-stable model $P$ of $C_K$ such that the conditions (1) (2) in the proof of Proposition~\ref{prop-rat-nm2} are satisfied. By the claim in the same proof, $U=\sm(P/S)\setminus \overline{\Delta_K}$. Thereby $C\setminus \Delta$ is an open subset of $U$. 
\end{remark} 

\begin{remark} In Proposition~\ref{prop-rat-nm2}, we can not drop the
  assumption that $S$ is excellent. For example, let $k$ be a field of
  characteristic $2$, and $K:=k(t, u^2)\subset k[[t]]\subset k((t))$
  with $u\in k[[t]]$ transcendental over $k(t)$. The discrete valuation on
$k((t))$ induces a discrete valuation on $K$, and let $R$ be the
corresponding (discrete) valuation ring (with $t\in R$ a
uniformizer). 
The completion $\widehat R$ of
$R$ is $k[[t]]\subset k((t))$, hence
$\widehat{K}=\mathrm{Frac}(\widehat{R})=k((t))$. 
Note that $u\in \widehat K$ is purely inseparable
  of degree $2$ over $K$ (in particular, $R$ is not excellent).
Consider $v=u^2\in K$ and 
let $U_K=\Spec(K[X_1,X_2]/(X_1^2-vX_0^2-X_0)$. This is the underlying scheme of 
a unipotent group. As $v\notin K^2$, $U_K\not\cong \mathbb A_K^1$. On the other 
hand,  $U_{\widehat{K}}\cong \mathbb{A}_{\widehat{K}}^{1}$. It follows that 
$U_{\widehat K}\cong \mathbb{G}_{a,\widehat K}$. Therefore,  by \cite{BLR}, 10.2/2, 
$U_K$ does not admit N\'eron lft-model over $S$. 
\end{remark}
\end{section}

\begin{section}{N\'eron lft-models of affine curves}\label{nm-rat-gl} 

The aim of this section is to prove the existence of N\'eron
lft-models of affine curves over $K$ different from the affine line 
(Theorem~\ref{nm-af}). 
Let $S$ be a Dedekind scheme with field of functions $K$. 

\begin{subsection}{Globalizing local N\'eron lft-models} 

\begin{lemma}\label{af-nm-global} {Let $S$ be a Dedekind scheme with $K$ its field of functions. Let $U_K$ be a separated connected smooth curve over $K$. 
Suppose that
\begin{enumerate}[{\rm (i)}]
\item for any closed point $s\in S$, $U_K$ admits a N\'eron lft-model 
$U(s)$ over $\cO_{S,s}$;
\item there exists a model of finite type $U^0$ of $U_K$ over $S$ such that for 
all $s\in S$, the isomorphism $U^0_K\to U(s)_K$ extends to an 
open immersion $U^0\times_S \Spec(\cO_{S,s})\to U(s)$. 
\end{enumerate}
Then $U_K$ admits a N\'eron lft-model over $S$. }
\end{lemma}

\begin{proof} The proof is inspired from that of \cite{BLR}, 10.1/7. 
The N\'eron lft-model of 
$U_K$ over $S$ will be obtained by gluing the local N\'eron lft-models
$U(s)$ for all closed points $s\in S$. 
We will first extend $U(s)$ to a model $V(s)$ over $S$ which coincides with
$U^0$ above $S\setminus \{ s\}$ and then glue the various $V(s)$
in a natural way to obtain the N\'eron lft-model $U$ of $U_K$ over $S$. 

Fix a closed point $s$. As $U(s)$ is locally of finite type, 
any connected component $U(s)_{s,\alpha}$ 
of $U(s)_s$ is open in the closed fiber $U(s)_s$, hence 
its union with $U_K$ is a quasi-compact open subset $U_\alpha(s)$ 
of $U(s)$. We can extend $U_\alpha(s)$ to a {separated} scheme of finite 
type $U_\alpha$ over $S$ (use \cite{EGA}, IV.8.10.5). As $U_\alpha$
and $U^0$ are both of finite type over $S$ 
and have the same generic fiber, they are $S$-isomorphic over a dense open 
subset $S_\alpha\subseteq S\setminus \{ s\}$ (and the $S$-isomorphism
is unique once the isomorphisms $U^0_K\to U_K$, $(U_{\alpha})_K\to U_K$
are fixed because $U_{\alpha}$ is separated over $S$). Now we glue the 
separated morphisms of finite type 
$$U_\alpha\times_S (S_\alpha\cup \{ s\})\to S_\alpha\cup \{ s\}, \quad 
U^0\times_S (S\setminus \{s\})\to S\setminus \{s\}$$ 
above $(S_\alpha\cup \{ s\})\cap (S\setminus \{s\})=S_\alpha$. 
The resulting $S$-scheme $V_\alpha$ is separated and of finite type because these
properties are satisfied above $S_\alpha\cup \{ s\}$ and $S\setminus \{ s\}$. 
By construction, we have canonically 
$$V_\alpha\times_S (S\setminus \{s\})=U^0\times_S
(S\setminus \{ s\}), \quad  V_\alpha\times_S \Spec(\cO_{S,s})=U_\alpha(s).$$
Next we glue the various $V_\alpha$ (when $U_\alpha(s)$ runs through
the connected components of $U(s)_s$) with the condition 
$V_\alpha\cap V_{\alpha'}=U^0$
if $\alpha\ne \alpha'$. The resulting $S$-scheme $V(s)$
satisfies canonically 
$$V(s)\times_S (S\setminus \{s\})=U^0\times_S
(S\setminus \{ s\}), \quad  V(s)\times_S \Spec(\cO_{S,s})=U(s).$$
Hence $V(s)$ is separated and locally of finite type over $S$. 
Moreover, Condition (ii) implies that the isomorphism $U^0_K\to V(s)_K=U(s)_K$ extends to an open immersion $U^0\to V(s)$. 

Finally, we glue the various $V(s)$ when $s$ runs through the closed
points of $S$ with the condition $V(s)\cap V(s')=U^0$ if $s\ne s'$. 
The resulting $S$-scheme $U$ is locally of finite type and 
$U\times_S \Spec(\cO_{S,s})\cong U(s)$ for all $s\in S$. 
By Corollary~\ref{neron-l2g}, $U$ is the N\'eron lft-model of $U_K$ over$S$, as required. 
\end{proof} 

\begin{lemma} \label{af-nm-global-2} Let $S$ be a Dedekind scheme. 
Let $X_K$ be a smooth connected separated curve over $K$
and let $U_K$ be an open dense subscheme of $X_K$. 
Suppose that $X_K$ has a smooth model $X$ over $S$ such that 
for all closed points $s\in S$, $X\times_S \Spec(\cO_{S,s})$ satisfies
the property in Theorem~\ref{nm-open-st}. Then $U_K$ admits 
a N\'eron lft-model over $S$. 
\end{lemma}

\begin{proof} Let $\Delta_K:=X_K\setminus U_K$, $X^0\subset X$ any
  quasi-compact open subset containing $X_K$, and
  $\Delta^0:=\overline{\Delta_K}\subset X^0$. By
  Theorem~\ref{nm-open-st}, the N\'eron model of $U_K$ over
  $\mathrm{Spec}(\cO_{S,s})$ exists for all closed points $s\in S$,
  and by Theorem~\ref{nm-open-st}~(2), $U^0:=X^0\setminus \Delta^0$
  verifies the hypothesis of Lemma~\ref{af-nm-global}. Thus we can
  apply Lemma~\ref{af-nm-global} to conclude. 
\end{proof}

\begin{proposition}\label{af-pg} Let $S$ be a Dedekind scheme. 
Let $X_K$ be a connected regular proper curve over $K$ of arithmetic 
genus $\ge 1$. Let 
$U_K$ be a dense open subscheme of $X_K$ contained in the smooth locus of $X_K/K$. Suppose either $S$ is excellent or the scheme $X_K$ is smooth over $K$. Then $U_K$ admits a N\'eron lft-model over $S$. 
\end{proposition}

\begin{proof} Let $X=X_\sm$ be the smooth locus of the minimal 
proper regular model of $X_K$ over $S$. Let $s\in S$ be a closed
point. Then $X\times_{S} \widehat{\cO_{S,s}^{\sh}}$ is the 
N\'eron model of $X_K$ over $\widehat{\cO_{S,s}^{\sh}}$ 
by Theorem~\ref{pj-nm} and because {in both cases} the minimal proper regular
as well as the its smooth locus commute with strict henselization
and completion (\cite{Liu}, 9.3.28). Applying Lemma~\ref{af-nm-global-2}
to $U_K\subseteq X_{K, \sm}$, we see that 
$U_K$ has a N\'eron lft-model $U$ over $S$. 
\end{proof}
\end{subsection}

\begin{subsection}{Weil restriction} 

\begin{proposition}[see also \cite{BLR}, 10.1/4]\label{neron-weil} Let $S$ be a Dedekind scheme with
field of functions $K$ and let $X_K$ be a separated smooth connected
  curve over $K$. Let $K'/K$ be a finite extension, and let $S'$ be the normalization
  of $S$ in $K'$. Assume that 
  \begin{enumerate}[{\rm (i)}]
\item $S'\to S$ is finite ({\it e.g.}, if $S$ is excellent or $K'/K$ is
  separable); 
  \item $X_{K'}$ admits a N\'eron lft-model (resp. N\'eron model) $X'$ over $S'$;
\item any quasi-compact open subset of $X'$ is quasi-projective over $S'$. 
\end{enumerate}
Then $X_K$ admits also a N\'eron lft-model (resp. N\'eron model) over $S$.  
\end{proposition}

\begin{proof} Let $s\in S$. Then any finite subset $F$ of $X'_{s}:=X'\times_S
\Spec k(s)$ is contained in an affine open subset of $X'$. Indeed, $F$
is contained in a quasi-compact open subset $W$ of $X'$. Let $V$ 
be an affine open neighborhood of $s$. As $S'\times_S V$ is finite over
$V$, $W\times_S V=W\times_{S'}(S'\times_S V)$ is quasi-projective over $V$, 
hence $F$ is contained in an affine open subset of $W\times_S V\subseteq X'$. 

The morphism $S'\rightarrow S$ is finite and locally
  free, hence the above property implies that the Weil restriction functor
  $Y:=\mathrm{Res}_{S'/S}X'$ is representable by a smooth $S$-scheme
  locally of finite type (\cite{BLR}, 7.6/4). Furthermore, by the
  functoriality of the Weil restriction, one checks easily
  that $Y$ is the $S$-N\'eron lft-model of its generic fiber $Y_K\cong
  \mathrm{Res}_{K'/K}(X_{K'})$. Finally, remark that as $X_K/K$ is
  separated, the adjunction map $X_K\rightarrow
  \mathrm{Res}_{K'/K}(X_{K'})=Y_K$ is a closed immersion, hence it
  suffices to apply Proposition~\ref{2nd-app-of-fY-sm} to conclude the existence
  of N\'eron lft-model or N\'eron model of $X_K$. 
\end{proof}

The following result is useful when we want to check the condition~(iii) of Proposition~\ref{neron-weil}. 

\begin{proposition} \label{quasi-pj}
{\rm ({Quasi-projectivity})} Let $S$ be a Dedekind scheme with field of functions $K$, and let $U$ be a connected regular relative curve over $S$ locally of finite type. Suppose that either $U_K$ has a smooth compactification or $S$ is excellent. Then any quasi-compact open subset 
of $U$ is quasi-projective over $S$ (in the sense of 
\cite{EGA}, II.5.3.1).   
\end{proposition}

\begin{proof} Let $U_0$ be a quasi-compact open subset of $U$. 
Let $U_0\subseteq U_0'$ be a Nagata compactification. 
The hypothesis on $U_K$ or $S$ implies
that there exists a desingularization morphism $Z\to U_0'$ 
which is an isomorphism {above} $U_0$. 
So $U_0$ is isomorphic to 
an open subscheme of a regular proper flat $S$-scheme $Z$. 
It is enough to show that $Z\to S$ is projective. 
This is a theorem of Lichtenbaum when $S$ is affine
(\cite{Lic}, Theorem 2.8 or \cite{Liu}, 8.3.16), but the
proof works exactly in the same way in the general case: find a
positive horizontal Weil divisor $H$ on $Z$ which meets all
irreducible components of all fibers of $Z\to S$. As $Z$ is regular, 
$H$ is defined by an invertible sheaf $\mathcal L$ on $Z$. The
hypothesis on $H$ implies that $\mathcal L$ is fiberwise
ample, hence $\mathcal L$ is relatively ample for $Z\to S$. 
\end{proof}

\begin{remark} If the Dedekind scheme $S$ is separated, then any quasi-projective 
scheme over $S$ is a subscheme of some $\mathbb P^N_S$. 
Indeed, $S$ then has an invertible ample sheaf 
(\cite[\href{http://stacks.math.columbia.edu/tag/09NZ}{Proposition
09NZ}]{Stacks}), and one can conclude with \cite{EGA}, II.5.3.3.  
\end{remark}
\end{subsection}

\begin{subsection}{Affine open subsets of a conic} \label{rat-curves} 

In this Subsection, we discuss the existence of N\'eron lft-models of an affine open subscheme $U_K$ of a smooth projective conic $C_K/K$. 
Observe first that $\mathbb A_K^1$ 
does not admit N\'eron lft-model over $S$ (\cite{BLR}, 10.1/8). 

\begin{proposition}\label{nm-rat-gl-main} Let $S$ be an excellent Dedekind scheme with field of functions $K$. Let 
$U_K$ be an affine open subscheme of a smooth projective conic $C_K$ over $K$.
Suppose $U_K$ is not isomorphic to $\mathbb A_K^1$. Then $U_K$ admits a N\'eron 
lft-model $U$ over $S$. 
\end{proposition}

\begin{proof} If over an algebraic closure $\overline{K}$ of $K$, 
${C}_{\overline{K}}\setminus U_{\overline{K}}$ contains at least
two points, then there exists a finite extension $K'/K$ such that 
$U_{K'}$ is isomorphic to an open subscheme of $\mathbb G_{m,K'}$. It
follows from Proposition~\ref{neron-weil} that we can suppose $U_K$ is an open
subscheme of $\mathbb G_{m, K}$. The latter has 
 a N\'eron lft-model $G$ over $S$, 
locally on $S$ compatible with any index $1$ extension 
(see the construction of \cite{BLR}, 10.1/5). 
Consequently, by Lemma~\ref{af-nm-global-2}, $U_K$ admits a N\'eron lft-model over $S$. For the rest of the proof, we can therefore suppose that $U_{\overline{K}}$ is
${C}_{\overline{K}}$ minus one point. 
So $\Delta_K:=C_K\setminus U_K$ consists of a single point $q_{\infty}$ which is purely inseparable of degree $> 1$ over $K$ because $U_K\not\cong \mathbb A^1_K$.  

As $C_K$ is smooth over $K$, there exists a separable extension $K'/K$ such that $C_{K'}\cong \mathbb P^1_{K'}$.
The point of $C_{K'}\setminus U_{K'}$ is still purely inseparable of degree $>1$ over $K'$
because $K'/K$ is separable. Using Proposition~\ref{neron-weil}, we can reduce to the case
$C_K=\mathbb P^1_K$. Let $P\cong \mathbb P^1_S$ be a smooth proper model of $\mathbb P^1_K$ over $S$, 
and let $\Delta=\overline{\{q_{\infty}\}}\subset P$. We know (Proposition~\ref{prop-rat-nm2})
that for all closed points $s\in S$, $U_K$ admits a N\'eron model $U(s)$ over $\cO_{S,s}$. 
To find a global N\'eron lft-model, it is enough to show that for $U^0:=P\setminus \Delta$,
the canonical morphism 
\begin{equation}\label{map-s} 
U^0\times_S \Spec(\cO_{S,s})\to U(s)
\end{equation}
is an open immersion for all {$s$ contained in a dense open subset $V\subset S$: the base change $U^0\times_S V$ satisfies then Condition~(ii) of Lemma~\ref{af-nm-global}.} As $S$ is excellent, so is $\Delta$. Thus the regular locus of $\Delta$ is 
open in $\Delta$. 
Shrinking $S$ if necessary, we can assume $\Delta$ is regular. Then 
for any closed point $s\in S$, the morphism \eqref{map-s} is an open immersion
by Proposition~\ref{prop-rat-nm2} and Remark~\ref{comp-to-prop-rat-nm2}, and 
the proposition is proved. 
\end{proof}

\begin{corollary}\label{app-conj-BLR} Let $S$ be an excellent Dedekind scheme of characteristic $p>0$, with $K$ its field of functions. Let $G_K$ be a connected smooth $K$-wound unipotent group of dimension $1$. Then $G_K$ admits a N\'eron lft-model over $S$. 
\end{corollary} 

\begin{remark} Let $S$ be a Dedekind scheme with field of functions
$K$. One can deduce from Proposition~\ref{nm-conic} and
Proposition~\ref{2nd-app-of-fY-sm} that if $X_K$ is a connected separated smooth
$K$-variety admitting N\'eron lft-model over $S$, then $X_K$ does not
contain any closed subscheme isomorphic to $\mathbb{P}_{K}^{1}$ or
$\mathbb{A}_{K}^{1}$. Conversely, if $X_K$ is the underlying scheme of a smooth commutative algebraic group over $K$, the latter condition is
also sufficient for the existence of N\'eron lft-model when $S$ is
local and excellent (\cite{BLR}, 10.2/2). 

When $S$ is global and excellent, whether this latter condition is
sufficient is still an open question. 
It is conjectured (\cite{BLR}, 10.3, Conjecture I)
that the answer is yes.  Some positive examples
are known in \cite{BLR}, Chap. 10. 
Corollary~\ref{app-conj-BLR} provides some evidence 
in favor of this conjecture. 
Together with the well-known
  results for abelian varieties and for tori, we deduce that when $S$
  is excellent, any
  smooth connected $K$-algebraic group $G_K$ of dimension one admits a
  N\'eron lft-model over $S$ if and only if $G_K$ is not isomorphic to
  $\mathbb G_{a,K}$. In other words, Conjecture I of \cite{BLR}, 10.3
  holds when $\dim(G_K)=1$. When the unipotent
  group scheme $G_K$ in Corollary~\ref{app-conj-BLR} admits a regular
  compactification of genus $\geq 1$, or equivalently when
  $\mathrm{uni}(G_K)=0$, Corollary~\ref{app-conj-BLR} is a special
  case of \cite{BLR}, 10.3/5. So the new case provided here  is when
  $\mathrm{uni}(G_{K})>0$, or equivalently when $G_K$ admits a smooth
  compactification of genus $0$. The latter happens only when
  $\mathrm{char}(K)=2$ (see last paragraph of \cite{BLR}, 10.3,
  p. 316). In this case, $G_K$ is the subgroup of
  $\mathbb{G}_{a,K}^2=\mathrm{Spec}(K[X,Y])$ defined by the equation
  $X^2=Y+aY^2$ for some $a\in K\setminus K^2$, which, as a scheme, is isomorphic to $\mathrm{Proj}(K[T,T'])\setminus V_{+}(T^2-aT'^2)$.  
\end{remark}
\end{subsection}

\begin{subsection}{N\'eron lft-models for affine curves}
We are now in the position to prove the existence of N\'eron lft-models for affine curves.

\begin{theorem}\label{nm-af} Let $S$ be an excellent Dedekind scheme with field of functions $K$. Let $U_K$ be an affine smooth connected curve over $K$. Then $U_K$  
admits a N\'eron lft-model over $S$ if $U_K\not\cong 
\mathbb A^1_{L}$ for any finite extension $L/K$. 
\end{theorem}

\begin{proof} By Lemma~\ref{connectedness}, we can suppose $U_K$
is  geometrically connected. If the regular compactification of
$U$ has positive arithmetic genus, then $U_K$ has a N\'eron 
lft-model by Proposition~\ref{af-pg}. Otherwise, $U_K$ is an affine open subset of a smooth projective conic over $K$, not isomorphic to $\mathbb A^1_K$. 
So $U_K$ admits a N\'eron lft-model over $S$ by Proposition~\ref{nm-rat-gl-main}.  
\end{proof}

Next we examine when the N\'eron lft-model is of finite type. 

\begin{proposition} \label{af-nm-ft} Let $S$ be an excellent  
Dedekind scheme with field of functions $K$. Let $U_K$ be an affine smooth 
geometrically connected curve of $K$, not isomorphic to $\mathbb A_K^1$, 
and let $C_K$ be its regular compactification. Denote by 
$\Delta_K:=C_K\setminus U_K$. 
Let $C$ be a relatively minimal regular model of $C_K$ over $S$ and let 
$\Delta$ be the reduced Zariski closure of $\Delta_K$ in $C$. Let $U$
be the N\'eron lft-model of $U_K$ over $S$. 
Then the following properties are true. 

\begin{enumerate}[{\rm (1)}] 
\item The scheme $U/S$ is of 
finite type if and only if $\Delta_K(K^{\sh}_s)=\emptyset$ for all
closed points $s\in S$ and if 
$\Delta_s\cap C_{\sm,s}(k(s)^{\mathrm{sep}})=\emptyset$ for almost all $s\in S$.
\item Assume $S$ is infinite. For each closed point $s \in S$, set  $U(s):=U\times_S \mathrm{Spec}(\cO_{S,s})$,  the local N\'eron lft-model of $U_K$ over $\mathrm{Spec}(\cO_{S,s})$.  Let $K^{\sep}$ denote a separable closure of $K$. 
\begin{enumerate}[{\rm (i)}]
\item $\Delta_K(K^{\sep})=\emptyset$ if and only if all the local N\'eron lft-models $U(s)$ are of finite type. 
\item If $\Delta_K(K^{\sep})\neq \emptyset$, the local N\'eron lft-models $U(s)$ are not of finite type for all but finitely many closed points $s\in S$. In particular,
if $K$ has characteristic $0$, then $U$ is never
of finite type.
\end{enumerate}
\end{enumerate}
\end{proposition}

\begin{proof} (1) {The $S$-scheme} $U$ is of finite type if and only if 
\begin{enumerate}[{{\hskip 12pt} \rm (a)}] 
\item for all closed point $s\in S$, $U(s)$ is of finite type 
over $\cO_{S,s}$; and if 
\item $U$ is of finite type over some open dense subset of $S$, or equivalently, $U_s$ is
connected for all but finitely many $s$ (Proposition~\ref{nm-ft}). 
\end{enumerate}
{Therefore we only need to show that the conditions of (1) are equivalent to the conditions (a) and (b) above.}

First assume that $C_K$ is of arithmetic genus $>0$. Then 
$U_K\subseteq X_K:=\sm(C_K/K)$. Let $X$ be the N\'eron model of 
$X_K$ over $S$, {equal to the smooth locus $C_{\sm}$ of $C/S$} (Theorem~\ref{pj-nm}). The minimal regular
model $C$ commutes with strict henselization and completion
(\cite{Liu}, 9.3.28), so by Theorem~\ref{nm-open-st}, 
for any {closed point} $s\in S$, $U(s)$ is of finite type over $\cO_{S,s}$  if 
and only if $\Delta_K(K_s^{\sh})=\emptyset$. Now consider the connectedness
at a closed point $s\in S$. Shrinking $S$ if necessary, we can suppose
$X_s{=C_{\sm,s}}$ is connected. If $\Delta_s\cap C_{\sm,s}(k(s)^{\mathrm{sep}})=\emptyset$,
then $U_s=X_s\setminus \Delta_s$ (Proposition~\ref{nm-open-st} (3)) is connected. 
Conversely, suppose $U_s$ is connected. By separatedness, 
$U_s\to X_s$ is an open immersion. Then over $\cO_{S,s}$, $U\to X$ is
an open immersion, as $X\setminus U$ contains $\Delta_K$, hence 
$U\subseteq X\setminus \Delta$. By (\ref{X-U}), we have 
$U=X\setminus \Delta$ (over $\cO_{S,s}$), thus 
$\Delta_s\cap C_{\sm,s}(k(s)^{\mathrm{sep}})=\emptyset$ (Theorem~\ref{nm-open-st}~(3)).

Now suppose $C_K$ is a smooth conic. Let $s\in S$ be a closed point.
If $\Delta_K(K_s^{\sh})=\emptyset$, then $U(s)$ is of finite type
over $\cO_{S,s}$ by Proposition~\ref{prop-rat-nm2}. 
Conversely, suppose $\Delta_K(K_s^{\sh})\ne\emptyset$. 
It is enough to show $U(s)$ is not of finite type over $\cO_{S,s}^{\sh}$. 
Thus we can suppose $K$ strictly henselian and $\Delta_K(K)\ne\emptyset$. If there are two
rational points in $\Delta_K$, then $U_K$ is isomorphic to an
open subscheme of $\mathbb G_{m,K}$, hence $U(s)$ is not of
finite type by Theorem~\ref{nm-open-st} (1). Otherwise, there exists a non-rational point $q_\infty\in \Delta_K$. We apply again 
Theorem~\ref{nm-open-st} (1) to $U_K\subset C_K\setminus\{ q_\infty\}$
to conclude that $U(s)$ is not of finite type. {So again $U(s)$ is of finite type if and only if $\Delta_K(K_s^{\sh})=\emptyset$}. 

To see the {equivalence between Condition (b) above and the second condition of (1)}, we are allowed to shrink $S$ and suppose 
that $C/S$ is smooth with connected fibers. The {condition} $\Delta_s\cap C_{\sm,s}(k(s)^{\mathrm{sep}})=\emptyset$ then implies that 
$C\setminus \Delta$ is the N\'eron model of $U_K$ over $\cO_{S,s}$
(see Claim in the proof of Proposition~\ref{prop-rat-nm2}). Thus $U=C\setminus \Delta$ is 
of finite type {with connected fibers} over $S$. Conversely, suppose $U/S$ is of 
finite type (up to replace $S$ by some open dense subset). Shrinking $S$ if necessary, the isomorphism 
$U_K\to (C\setminus \Delta)_K$ extends to an isomorphism $U\cong C\setminus \Delta$. If there exists $x\in \Delta_s\cap C_{\sm,s}(k(s)^{\sep})$, then there exists 
$y\in U_K(K_s^{\sh})=(C\setminus \Delta)_K(K_{s}^{\sh})$ which 
does not specializes to $U_s$. Contradiction 
with the N\'eron mapping property of $U$. 

(2) As $S$ is infinite, $\Delta_K(K^{\sep})\neq\emptyset$ if and only if $\Delta_K(K_{s}^{\sh})\neq\emptyset$ for some (hence for almost all) $s\in S$. Thereby (2) follows directly from (1).  
\end{proof}

\begin{remark}\label{not-ft} In general it is not true that if $\Delta_K(K^{\sep})=\emptyset$ then $U$ is of finite type. One can construct examples similar to that of
  Oesterl\'e (\cite{BLR}, 10.1/11), showing that the existence of
  local N\'eron models does not imply the existence of global N\'eron
  model.  Let $S$ be an excellent Dedekind scheme of characteristic
  $p>0$ with infinitely many closed points, such that each closed
  point of $S$ has perfect residue field (for example, take $S$ a
  smooth algebraic curve over a perfect field of
  characteristic $p$). Let $q_{\infty}\in \mathbb P_{K}^{1}$ be a purely
  inseparable closed point of degree $>1$. Let $U_K=\mathbb 
  P^1_K\setminus \{  q_\infty\}$ and let $U$ be its $S$-N\'eron 
model. Then $U\times_S \Spec(\cO_{S,s})$ is of finite type. But as $U_s$ has 
  two connected components for almost all $s\in S$, 
$U$ is not of finite type over $S$ by Proposition~\ref{nm-ft} (3). 
\end{remark}
\end{subsection}
\end{section}

\end{document}